\documentclass[11pt,reqno]{amsart}

 \usepackage{amsfonts,amsmath,amscd,amssymb,amsbsy,amsthm,amstext,amsopn,amsxtra,mathtools}
 \usepackage{color,fullpage,mathrsfs}
 \usepackage{subfigure}
 \usepackage{verbatim} 
 \usepackage{stmaryrd}
 \usepackage{extarrows}
 \usepackage[all]{xy}
 \usepackage{bm}   
 \usepackage{hyperref}
 \usepackage{tikz}
 \usepackage{tikz-cd}
 \usepackage{enumitem}
 \usepackage{multirow}

\usepackage[mathscr]{euscript}

 \usetikzlibrary{calc}
 \usetikzlibrary{cd}

 \hypersetup{colorlinks,linkcolor=[rgb]{0,0,1},citecolor=[rgb]{1,0,0}}
 
\usepackage{cancel}

 \numberwithin{equation}{section}
\let\nc\newcommand
\let\renc\renewcommand

\theoremstyle{plain}
\newtheorem{thm}{Theorem}
\newtheorem{prop}[thm]{Proposition}
\newtheorem{cor}[thm]{Corollary}
\newtheorem{lem}[thm]{Lemma}
\newtheorem{conjecture}[thm]{Conjecture}

\newtheorem{cond}[thm]{Condition}
\theoremstyle{definition}
\newtheorem{defn}[thm]{Definition}
\newtheorem{example}[thm]{Example}

\newtheorem{remark}[thm]{Remark}
\newtheorem{rem}[thm]{Remark}

\numberwithin{thm}{section}

\makeatletter
\renewcommand{\subsection}{\@startsection{subsection}{2}{0pt}{-3ex
plus -1ex minus -0.2ex}{-2mm plus -0pt minus
-2pt}{\normalfont\bfseries}} \makeatother

\numberwithin{equation}{section}

\def\gbf#1{\mbox{$\mathbf{#1}$}}

\DeclareMathOperator{\res}{{\mathrm{res}}}

\DeclareMathOperator{\Proj}{\mathrm{Proj}}

\DeclareMathOperator{\End}{\mathrm{End}}

\DeclareMathOperator{\gr}{\mathrm{gr}}

\DeclareMathOperator{\Lie}{\mathrm{Lie}}

\DeclareMathOperator{\Rep}{\mathrm{Rep}}

\newcommand{\bv}{\mathbf{v}} 

\newcommand{\beq}{\begin{equation}\label}
\newcommand{\eeq}{\end{equation}}

\DeclareMathOperator{\Spec}{\mathrm{Spec}}

\DeclareMathOperator{\Hom}{\mathrm{Hom}}
\DeclareMathOperator{\GL}{\mathrm{GL}}

\DeclareMathOperator{\Wt}{\mathrm{Wt}}

\nc{\Z}{\mathbb{Z}}
\newcommand{\N}{\mathbb{N}}
\newcommand{\Q}{\mathbb{Q}}
\newcommand{\R}{\mathbb{R}}
\newcommand{\C}{\mathbb{C}}

\nc{\rank}{\textrm{rank} \,}
\nc{\ds}{\dots}

\let\mc\mathcal
\let\mf\mathfrak

\nc{\mbf}{\mathbf}
\nc{\Res}{\mathsf{Res} \, }
\nc{\Ind}{\mathsf{Ind} \, }
\nc{\cont}{\textrm{cont}}

\nc{\msf}{\mathsf}

\nc{\minusone}{-1}
\nc{\minustwo}{-2}
\nc{\Mod}{\mathrm{Mod} \,}
\nc{\ms}{\mathscr}
\nc{\Frac}{\mathrm{Frac} \,}
\nc{\ra}{\rightarrow}
\nc{\hra}{\hookrightarrow}
\nc{\lab}{\label}
\renc{\O}{\mc{O}}
\nc{\Tan}{\mc{T}}
\nc{\ul}{\underline}
\nc{\s}{\mathfrak{S}}
\nc{\g}{\mf{g}}
\nc{\pa}{\partial}
\nc{\tit}{\textit}
\nc{\Maxspec}{\mathrm{Maxspec} \, }
\nc{\gldim}{\mathrm{gl.dim}}
\nc{\rkm}{\mathrm{rk} \, (\mf{m})}
\nc{\sm}{\mathrm{sm}}
\nc{\PD}{\mathbb{PD}}
\nc{\hilb}{\textrm{Hilb}}
\nc{\<}{\langle}
\renc{\>}{\rangle}
\nc{\T}{\mathbb{T}}
\nc{\X}{\mathbb{X}}
\nc{\F}{\mathbb{F}}
\nc{\id}{\msf{id}}
\nc{\A}{\mathbb{A}}
\nc{\Grat}{\mc{Grat}}
\nc{\Squo}[1]{\A^{(#1)}}
\nc{\twist}{\mathrm{twist}}
\nc{\Cd}{\mc{C}}
\nc{\Span}{\mathrm{Span}}
\nc{\Grass}{\mathrm{Gr}}
\nc{\Fr}{\mathrm{Fr}}
\nc{\pco}[1]{k[V]^{p\mathrm{co} #1}}
\nc{\Irr}{\mathrm{Irr}}
\renc{\o}{\otimes}
\renc{\gr}{\mathsf{gr}}
\nc{\fin}{\mathrm{fin}}
\nc{\aff}{\mathrm{aff}}
\nc{\algD}{\mf{D}}
\nc{\hr}{\mf{h}_{\textrm{reg}}}
\nc{\D}{\mathscr{D}}
\nc{\PIdeg}{\mathrm{P.I.-degree}}
\nc{\ch}{\mathrm{ch}}
\nc{\ev}{\mathsf{ev}}
\nc{\Stab}{\mathrm{Stab}}
\nc{\Der}{\mathrm{Der}}
\nc{\rightsim}{\stackrel{\sim}{\longrightarrow}}
\nc{\HZ}{H_{\mbf{h},\Z}(\Z_m)}
\nc{\sing}{\mathrm{sing}}
\nc{\dd}{\mathscr{D}}
\nc{\bc}{\mathbf{c}}
\nc{\vc}{\underline{\mathbf{c}}}
\nc{\ba}{\mathbf{a}}
\nc{\reg}{\mathrm{reg}}
\nc{\Amp}{\mathrm{Amp}}
\nc{\Nef}{\mathrm{Nef}}
\nc{\SL}{\mathrm{SL}}
\nc{\SO}{\mathrm{SO}}
\nc{\Sp}{\mathrm{Sp}}
\nc{\Sym}{\mathrm{Sym}}
\nc{\Mov}{\mathrm{Mov}}
\nc{\Pic}{\mathrm{Pic}}
\nc{\Cs}{\C^{\times}}
\nc{\Nak}[3]{\mf{M}_{{#1}} ({#2},{#3}) }
\nc{\Naka}[2]{\mf{M}({#1},{#2}) }
\nc{\Mtheta}[1]{\mc{M}_{#1}}
\DeclareMathOperator{\Seshadri}{\mathrm{S}} 
\DeclareMathOperator{\head}{hd}
\DeclareMathOperator{\tail}{tl}
\nc{\bw}{\mathbf{w}}

\nc{\bn}{\mathbf{n}}
\nc{\CB}{\mathrm{CB}}
\nc{\GVect}{\Lambda}
\nc{\pZ}{\overline{Z}}
\nc{\Qu}{Q}
\nc{\Supp}{\mathrm{Supp}}
\nc{\mr}{\mathrm}

\newcommand{\one}{\ensuremath{(\mathrm{i})}}
\newcommand{\two}{\ensuremath{(\mathrm{ii})}}
\newcommand{\three}{\ensuremath{(\mathrm{iii})}}
\newcommand{\four}{\ensuremath{(\mathrm{iv})}}

\newcommand{\CC}{\ensuremath{\mathbb{C}}} 
\newcommand{\kk}{\ensuremath{\Bbbk}} 
 
\newcommand{\QQ}{\ensuremath{\mathbb{Q}}}

\newcommand{\ZZ}{\ensuremath{\mathbb{Z}}}

\newcommand{\git}{\ensuremath{/\!\!/\!}}
\DeclareMathOperator{\Cox}{Cox}
\DeclareMathOperator{\Exc}{Exc}
\DeclareMathOperator{\Uns}{Uns}
\newcommand{\Xtheta}[1]{X_{\theta{#1}}}

\nc{\NS}{\operatorname{NS}}

\nc{\0}{\operatorname{0}}
\nc{\I}{\operatorname{I}}
\nc{\II}{\operatorname{II}}
\nc{\III}{\operatorname{III}}

\begin{document}

\title{Birational geometry of quiver varieties \\ and other GIT quotients}

\dedicatory{In memory of Tom Nevins}

\author{Gwyn Bellamy}
\address{School of Mathematics and Statistics, University of Glasgow, 15 University Gardens, Glasgow, G12 8QW.}
\email{gwyn.bellamy@glasgow.ac.uk} 
\urladdr{http://www.maths.gla.ac.uk/~gbellamy/}

\author{Alastair Craw} 
\address{Department of Mathematical Sciences, 
University of Bath, 
Claverton Down, 
Bath BA2 7AY, 
United Kingdom.}
\email{a.craw@bath.ac.uk}
\urladdr{http://people.bath.ac.uk/ac886/}

\author{Travis Schedler} 
\address{Imperial College London, Huxley Building,
South Kensington Campus, London SW7 2AZ}
\email{t.schedler@imperial.ac.uk}
\urladdr{https://www.imperial.ac.uk/people/t.schedler}



\maketitle

\begin{abstract}
    We prove that all projective crepant resolutions of Nakajima quiver varieties satisfying natural conditions are also Nakajima quiver varieties.  More generally, we classify the small birational models of many Geometric Invariant Theory (GIT) quotients by introducing a sufficient condition for the GIT quotient of an affine variety $V$ by the action of a reductive group $G$ to be a relative Mori Dream Space. Two surprising examples illustrate that our new condition is optimal. When the condition holds, we show that the linearisation map identifies a region of the GIT fan with the Mori chamber decomposition of the relative movable cone of $V\git_\theta G$. 
    If $V\git_\theta G$ is a crepant resolution of $Y\!\!:= V\git_0 G$, then every projective crepant resolution of $Y$ is obtained by 
    varying $\theta$. Under suitable conditions, we show that this is the case for quiver varieties and hypertoric varieties. Similarly, for any finite subgroup $\Gamma\subset \SL(3,\CC)$ whose nontrivial conjugacy classes are all junior, we obtain a simple geometric proof of the fact that every projective crepant resolution of $\mathbb{C}^3/\Gamma$ is a fine moduli space of $\theta$-stable $\Gamma$-constellations. 
\end{abstract}

\section{Introduction}
Nakajima quiver varieties~\cite{Nak1994, NakDuke98} provide a rich source of examples illustrating many beautiful phenomena in algebraic geometry and geometric representation theory.  To recall the construction, consider a finite graph with vertex set $I$, and vectors $\mathbf{v}, \mathbf{w}\in \mathbb{N}^{I}$. This combinatorial data determines a Hamiltonian action of the group $G:= \prod_{i\in I} \GL(v_i)$ on a symplectic vector space $\mathbf{M}(\mathbf{v}, \mathbf{w})$, giving rise to a moment map $\mu\colon \mathbf{M}(\mathbf{v}, \mathbf{w})\to \mathfrak{g}^*$. For any character $\theta \in G^\vee$, the Nakajima \emph{quiver variety} is defined to be the Geometric Invariant Theory (GIT) quotient 
\begin{equation}
\label{eqn:quivervariety}
\mathfrak{M}_\theta(\mathbf{v},\mathbf{w}):= \mu^{-1}(0)\git_\theta \, G.
\end{equation}

 Under suitable conditions on $\bv, \bw$, and for any sufficiently general $\theta$ (see Section~\ref{sec:AppstoQuiverVars}), the 
 structure morphism $\mathfrak{M}_\theta(\mathbf{v},\mathbf{w})\to \mf{M}_0(\bv,\bw)$ is a crepant resolution of singularities.
 It follows from the work of Birkar, Cascini, Hacon and McKernan~\cite[Corollary~1.3.2]{BCHM} that $\mathfrak{M}_\theta(\mathbf{v},\mathbf{w})$ is a relative Mori Dream Space (see Namikawa~\cite{NamikawaMDS} or \cite[Lemma~5.3]{BCRSW21}).
Put simply, the birational geometry of $\mathfrak{M}_\theta(\mathbf{v},\mathbf{w})$ is especially well-behaved. It is therefore natural to ask for a concrete description of the relative movable cone and the set of all projective crepant resolutions of $\mathfrak{M}_0(\mathbf{v},\mathbf{w})$.

Here, we answer these questions in full 
 by proving directly that quiver varieties are relative Mori Dream Spaces, and we establish that every projective crepant resolution of $\mathfrak{M}_0(\mathbf{v},\mathbf{w})$ is itself a quiver variety. In doing so, we provide a vast generalisation of the results of \cite{BellamyCraw}, bypassing the \'{e}tale-local description of quiver varieties required there. Our proof does not use results from \cite{BCHM}, nor does it apply the relative version of the sufficient condition to be a Mori Dream Space given by Hu and Keel~\cite[Theorem~2.3]{HuKeel00}, because that condition does not apply even to the simplest quiver variety, namely, the minimal resolution of the $A_1$ surface singularity.
 
 \subsection{The main result for GIT quotients}
 
 In fact, our approach is much more general, and applies to a broad class of quotients that arises across algebraic geometry. Consider the action of a reductive group $G$ on an affine variety $V$. We do not assume that $V$ is normal.  The vector space $G^\vee_\QQ=G^\vee \otimes_\ZZ \QQ$ of rational characters decomposes into a polyhedral wall-and-chamber structure called the GIT fan. The set of generic stability parameters $\theta$ in $G^\vee_\QQ$ decomposes as the union of finitely many GIT chambers, each of which is the interior of a top-dimensional cone in the GIT fan. Our main result introduces a new sufficient condition guaranteeing that, for any chamber $C$ and any $\theta\in C$, the GIT quotient $\Xtheta{}:=V\git_\theta \, G$ is a relative Mori Dream Space over the affine quotient $Y:= V\git_0 \, G$; in fact, we describe a region of the GIT fan that captures completely the birational geometry of $\Xtheta{}$ over $Y$.
 
 Before stating our main result for GIT quotients, we describe our sufficient condition in general terms (see Condition~\ref{cond:GIT} for details). Recall that each character $\zeta\in G^\vee$ determines a $G$-linearisation of the trivial bundle on $V$ that descends to a line bundle $L_\zeta$ on $\Xtheta{}$ for generic $\theta\in G^\vee_{\QQ}$. Let $C$ be the chamber containing $\theta$. The \emph{linearisation map} for $C$ is the map of rational vector spaces
 \[
 L_C\colon G^\vee_\QQ\longrightarrow \Pic(\Xtheta{}/Y)\otimes_\ZZ \QQ
 \]
 defined by $L_C(\zeta) = L_\zeta$. One of the sufficient conditions from \cite[Lemma~2.2(4)]{HuKeel00} requires that $L_C$ is an isomorphism, and we impose this as part of the first criterion in our sufficient condition.

The novel aspect in our Condition~\ref{cond:GIT} is that the second and third criteria are phrased in terms of wall-crossing. For each GIT chamber $C$, we define a closed cone $R_C$ in $G^\vee_\QQ$ to be the union of the closures of a collection of chambers (see Definition~\ref{def:GITregion}), and our second and third criteria guarantee that \one\ variation of GIT quotient across each wall in the interior of $R_C$ induces a flip $V\git_{\theta_-} G\dashrightarrow V\git_{\theta_+}G$; and that \two\ the morphisms induced by variation of GIT quotient into each boundary wall from the interior of $R_C$ contract a divisor. We provide examples to show that even when the linearisation map $L_C$ is an isomorphism, it can happen that an interior GIT wall of $R_C$ does not induce a flip (see Example~\ref{exa:toricQuiverGIT}), and moreover, that even when $L_C$ is an isomorphism and all interior walls induce flips, it can happen that a boundary wall induces a morphism that does not contract a divisor (see Example~\ref{exa:threefoldflop}). Thus, all three criteria from Condition~\ref{cond:GIT} must be imposed to obtain the geometric behaviour that we seek.

The importance of our sufficient condition is illustrated by our main result for GIT quotients that can be stated as follows (see Theorem~\ref{thm:movable} and Corollaries~\ref{cor:sbm}-\ref{cor:mds}):
     \begin{thm}
     \label{thm:mainGITintro}
     For the action of a reductive group $G$  on an affine variety $V$, suppose that a GIT chamber $C$ in $G^\vee_\QQ$ satisfies Condition~\ref{cond:GIT}. For $\theta\in C$, write $\Xtheta{}:=V\git_\theta \, G$. 
 \begin{enumerate}
     \item[\one] The linearisation map is an isomorphism that identifies the GIT wall-and-chamber structure in $R_C$ with the decomposition of the movable cone $\Mov(\Xtheta{}/Y)$ into Mori chambers; 
     \item[\two] For generic $\zeta \in R_C$, the 
      GIT quotient $X_\zeta := V\git_\zeta G$ is  the birational model of $\Xtheta{}$ determined by the line bundle $L_C(\zeta)$; 
    \item[\three] 
     The GIT quotient $\Xtheta{}$ is a Mori Dream Space over $Y$. 
 \end{enumerate}
In particular, the $\QQ$-factorial small birational models of $\Xtheta{}$ are the birational models $V\git_\zeta \, G$ determined by the line bundles $L_C(\zeta)$ for generic $\zeta \in R_C$.
     \end{thm}
 
\begin{rem} 
As a consequence of the proof of the theorem, we see that the line bundles yielding birational models which are small, in the sense that they coincide with $\Xtheta{}$ outside codimension two, are precisely the $L_C(\zeta)$ for $\zeta$ in the interior of $R_C$. The corresponding model is the normalisation $\widetilde{X}_\zeta$ of the GIT quotient $X_\zeta$.
\end{rem}

 Thus, when Condition~\ref{cond:GIT} applies, Theorem~\ref{thm:mainGITintro} shows that the birational geometry of $\Xtheta{}$ over $Y$ is determined completely by variation of GIT quotient within the cone $R_C$. In fact, Theorem~\ref{thm:mainGITintro}\three\ implies that Condition~\ref{cond:GIT} is a new sufficient condition for a GIT problem to define a relative Mori Dream Space. Our approach does not use in any way the deep geometric results in the minimal model programme from \cite{BCHM}, relying instead on GIT arguments. 
    
    The work of Hu and Keel~\cite{HuKeel00} shows that for any Mori Dream Space $X$, Theorem~\ref{thm:mainGITintro} applies for the action of an algebraic torus on $\Spec \Cox(X)$; more generally,  \cite[Theorem~4.3.3.1]{ADHL} reconstructs a larger class of varieties from a quasi-torus action on $\Spec \Cox(X)$. However, we are particularly interested in examples where the reductive group $G$ need not be a (quasi-)torus,
    and where the affine variety $V$ is not the spectrum of $\Cox(X)$. In short, the Cox ring does not have a monopoly on finitely generated $\kk$-algebras that encode perfectly the birational geometry of a Mori Dream Space. 
 
 \subsection{Application to quiver varieties}
 \label{sec:AppstoQuiverVars}
  While our Condition~\ref{cond:GIT}  is strong enough to establish Theorem~\ref{thm:mainGITintro}, it is also weak enough to apply in a number of interesting situations. The case of primary interest to us is the group action that defines a quiver variety. 
  
  As above, for any graph with vertex set $I$, choose dimension vectors $\mathbf{v}, \mathbf{w}\in \mathbb{N}^{I}$ with $\mathbf{w}\neq 0$ and $v_i \neq 0$ for all $i\in I$.  For any $\theta\in G^\vee$,  the quiver variety $\mf{M}_\theta(\bv,\bw)$ is the GIT quotient from \eqref{eqn:quivervariety}. We assume throughout that there exists a simple representation of the doubled quiver in $\mu^{-1}(0)$, or equivalently, the vector $\alpha:=(1,\mathbf{v})\in \mathbb{N}\times \mathbb{N}^{I}$ satisfies Crawley-Boevey's condition $\alpha\in \Sigma_0$ (see Definition~\ref{def:Sigma0}). It follows that the zero fibre of the moment map $V:= \mu^{-1}(0)$ is an affine variety \cite[Theorem~1.2]{CBmomap}, and moreover, if $\theta$ is generic, then the projective morphism $\mathfrak{M}_\theta(\mathbf{v},\mathbf{w})\rightarrow  \mathfrak{M}_0(\mathbf{v},\mathbf{w})$ to the affine GIT quotient is a crepant resolution of singularities.
  
  For any GIT chamber $C$ and for $\theta\in C$, the quiver variety $\mathfrak{M}_\theta(\mathbf{v},\mathbf{w})$ is nonsingular and hence the linearisation map $L_C$ is surjective by the work of McGerty and Nevins~\cite[Theorem~1.2]{QuiverKirwan}.
 The assumption $\alpha\in \Sigma_0$ implies that $\mathfrak{M}_{\theta_0}(\mathbf{v},\mathbf{w})$ is singular for any non-generic $\theta_0\in G^\vee_\QQ$, so the morphism 
 \begin{equation}
 \label{eqn:introVGIT}
 \tau\colon \mf{M}_\theta(\bv,\bw)\longrightarrow \mf{M}_{\theta_0}(\bv,\bw)
 \end{equation}
 obtained by varying $\theta$ into the boundary of the chamber $C$ necessarily contracts at least one curve. This allows us to prove that $L_C$ is actually an isomorphism; we provide examples to show that $L_C$ need not be injective when $\alpha\not\in \Sigma_0$ (see Remark~\ref{rem:nontrivialkerLC}). The second and third criteria in our Condition~\ref{cond:GIT} are phrased in terms of wall crossing for quiver varieties, and for these, we control the dimension of the unstable locus of the morphism $\tau$ from \eqref{eqn:introVGIT} by analysing the singular locus of $\mathfrak{M}_{\theta_0}(\mathbf{v},\mathbf{w})$ for non-generic $\theta_0\in G^\vee_\QQ$. We distinguish flipping and divisorial contractions using the fact that $\tau$ is semi-small, a result due to Kaledin~\cite{Kaledinsympsingularities}. 
 
 This leads to our main result for quiver varieties (see Theorem \ref{thm:mainquiverVars} and Proposition \ref{prop:Namikawa-Weyl}). Note that all Nakajima quiver varieties are normal thanks to \cite{CBnormal} (see also \cite{BellSchedQuiver} for the nonaffine case).

 \begin{thm}
 \label{thm:mainquivertwointro}
 Under the above assumptions, the following hold:
 \begin{enumerate}
 \item[\one] every GIT chamber $C$ satisfies Condition~\ref{cond:GIT}, so Theorem~\ref{thm:mainGITintro} holds for the quiver variety $\Xtheta{}:=\mathfrak{M}_\theta(\mathbf{v},\mathbf{w})$ with $\theta\in C$; and
 \item[\two] for any chamber $C$, the GIT region $R_C$ is a simplicial cone that provides a fundamental domain for the action of the Namikawa--Weyl group on $G^\vee_{\QQ}$.
\end{enumerate}
 Thus, projective crepant resolutions $\mathfrak{M}_\theta(\mathbf{v}, \mathbf{w})
\to \mf{M}_0(\bv,\bw)$, taken up to isomorphism over $\mathfrak{M}_0(\mathbf{v},\mathbf{w})$ are in bijection with GIT chambers in $G^\vee_{\QQ}$ modulo the action of the Namikawa--Weyl group.
 \end{thm}

 This theorem provides a broad generalisation of the geometric interpretation by Kronheimer~\cite{Kronheimer} of the McKay correspondence, in which, for any finite subgroup $\Gamma\subset \SL(2,\kk)$, the minimal resolution of the Kleinian singularity $\mf{M}_0(\bv,\bw)\cong \mathbb{A}^2/\Gamma$ is constructed by  variation of GIT (or hyperk\"ahler) quotient as a quiver variety $\mathfrak{M}_\theta(\mathbf{v},\mathbf{w})$ for generic $\theta$, and moreover, any Weyl chamber of finite type ADE can be identified with the ample cone of the minimal resolution.

 Theorem~\ref{thm:mainquivertwointro} provides a direct, geometric proof of the fact that every quiver variety $\mathfrak{M}_\theta(\mathbf{v},\mathbf{w})$ is a relative Mori Dream Space over $\mathfrak{M}_0(\mathbf{v},\mathbf{w})$. In fact, we go further for quiver varieties by describing explicitly the hyperplane arrangement that determines the GIT chamber decomposition appearing in Theorem~\ref{thm:mainquivertwointro} (see Theorem~\ref{thm:GITquiverarrangement}).
 
 \begin{cor} 
 \label{cor:introNamikawaWeyl}
  Under the above assumptions, every projective crepant resolution of the affine quotient $\mathfrak{M}_0(\mathbf{v},\mathbf{w})$ is itself a quiver variety $\mathfrak{M}_\theta(\mathbf{v},\mathbf{w})$ for some generic $\theta\in G^\vee$.
  \end{cor}

 This result implies that every relative minimal model of a quiver variety is itself a quiver variety. This generalises and unifies the results from Yamagishi~\cite[Section~5]{Yamagishimasters}, and a pair of results of the authors \cite[Theorem~1.2]{BellamyCraw}, \cite[Theorem~1.2]{BCRSW21} (see Remark \ref{r:BC-case} for more details about the former).

 \subsection{Hypertoric varieties}
Our methods apply equally well to nonsingular hypertoric (originally called toric hyperk\"ahler) varieties. Here, a hypertoric variety is a variety obtained as a Hamiltonian reduction of a vector space by an algebraic torus, by which we mean a GIT quotient of the zero fibre of the moment map by the torus. In this case, the verification of our Condition~\ref{cond:GIT} for the standard GIT construction of a nonsingular hypertoric variety $X$ was largely carried out by Konno~\cite[Theorem~6.3]{Konno03}, though we also use the tilting bundle on $X$ constructed by \v{S}penko and Van den Bergh~\cite{SVdBhypertoric} (see also \cite{McBreenWebesterhypertoric}) to deduce that the linearisation map is surjective. Our main result for nonsingular hypertoric varieties, given in Theorem~\ref{thm:hypertoric}, establishes the following result.
 
 \begin{thm}
 \label{thm:conjecturefortori}
 Condition \ref{cond:GIT} holds for nonsingular hypertoric varieties.
 \end{thm}
 
 In this context, Theorem~\ref{thm:mainGITintro} implies in particular that every projective crepant resolution of a hypertoric variety is itself a hypertoric variety.

  \subsection{Application to threefold quotient singularities}
  The holomorphic symplectic nature of Nakajima quiver varieties and hypertoric varieties plays a role in our proofs of Theorem~\ref{thm:mainquivertwointro} and~\ref{thm:conjecturefortori}.  However, to emphasise that this is not an essential feature when applying Theorem~\ref{thm:mainGITintro}, we also study a class of examples in odd dimension.
  
  Consider any threefold quotient singularity of the form  $\mathbb{A}^3/\Gamma$, where $\Gamma\subset \SL(3,\kk)$ is a finite subgroup for which every nontrivial conjugacy class is \emph{junior} in the sense of Ito and Reid~\cite{ItoReid96} and $\kk$ is an algebraically closed field of characteristic zero. This condition is equivalent to requiring that any projective crepant resolution $f\colon X\to Y:=\mathbb{A}^3/\Gamma$ has all fibres of dimension at most one. One such resolution is provided by $X:= \Gamma\text{-Hilb}(\mathbb{A}^3)$, the fine moduli space of $\Gamma$-clusters in $\mathbb{A}^3$, as in \cite{BKR}, for which there is a natural GIT quotient construction $X\cong \Xtheta{} = V\git_\theta \, G$ for some generic $\theta$. The fact that $X$ contains no proper surfaces allows us to show that Condition~\ref{cond:GIT} holds in this setting, so the conclusions of Theorem~\ref{thm:mainGITintro} hold for the given GIT quotient description of $\Xtheta{}$ (see Theorem~\ref{thm:noseniormain}). Thus, we obtain:
 
 \begin{thm}
 \label{thm:mainintrothreefolds}
 Let $\Gamma\subset \SL(3,\kk)$ be a finite subgroup such that every non-trivial conjugacy class of $\Gamma$ is junior. Then every projective crepant resolution of $\mathbb{A}^3/\Gamma$ is a fine moduli space $\mathcal{M}_\theta$ of $\theta$-stable $\Gamma$-constellations for some generic $\theta\in \Theta$.
 \end{thm}
 
  Our direct and simple geometric proof of this result bypasses the algebraic approach via mutation from Nolla de Celis and Sekiya~\cite[Corollaries~1.3 and 1.5]{NollaSekiya17} that was pioneered by Wemyss and later generalised in his beautiful paper~\cite[Theorem~6.2]{Wemyss18}. Our Theorem~\ref{thm:mainGITintro}\one\ also provides a direct GIT description of the relative movable cone $\Mov(X/Y)$ in this setting.

 Very recently, Yamagishi~\cite[Theorem~1.1]{YamagishiII} announced that the conclusion of Theorem~\ref{thm:mainintrothreefolds} holds for any finite subgroup $\Gamma$ of $\SL(3,\kk)$. While the scope of our Theorem~\ref{thm:mainintrothreefolds} is much more limited, our approach is elementary: we show that there are no GIT walls of `type $\0$', and also, we do not require the deep results from \cite{BCHM}. Put simply, those $\Gamma$ for which every nontrivial conjugacy class is junior provide an especially simple family of examples that is amenable to our geometric approach.

\subsection{Optimality of Condition~\ref{cond:GIT}}
In the course of proving that our three main classes of examples satisfy Condition~\ref{cond:GIT}, we find that Condition~\ref{cond:GIT}(1) implies the other two conditions in those cases.  However, Examples \ref{exa:toricQuiverGIT} and \ref{exa:threefoldflop} show that this is not true in general, and indeed, for those two examples the conclusions of Theorem~\ref{thm:mainGITintro} do not hold.

The cases of Nakajima quiver varieties and hypertoric varieties are both Hamiltonian reductions associated to a representation of a reductive group.  In this  situation, the conclusion of Theorem~\ref{thm:mainGITintro} was anticipated (without proof) in \cite[Remark 2.21]{BPWAst}, assuming only Condition~\ref{cond:GIT}.(1). This omission would be explained if, in this setting, the first part of Condition~\ref{cond:GIT} implies the others: 

\begin{conjecture}
\label{conj:hamiltonianreduction}
Let $W$ be a linear representation of a complex reductive algebraic group $G$ and let $\mu\colon T^*W \to \mathfrak{g} := \Lie G$ be the associated moment map.  
Let $V := \mu^{-1}(0)$.  If, for some generic $\theta \in G^\vee_{\QQ}$, the VGIT morphism $X_\theta \to X_0$ is a crepant resolution, and the associated linearisation map $L_C$ is an isomorphism, then Condition~\ref{cond:GIT} holds.
\end{conjecture}

A proof of this conjecture would establish that Hamiltonian reductions bypass the more subtle and surprising VGIT phenomena exhibited by Examples \ref{exa:toricQuiverGIT} and \ref{exa:threefoldflop}.
More precisely, we anticipate that parts (2) and (3) of Corollary~\ref{c:useful-criterion} hold in the case of the Hamiltonian reductions appearing in Conjecture~\ref{conj:hamiltonianreduction}.

\medskip

\noindent \textbf{Notation.\ } 
Let $\kk$ be an algebraically closed field of characteristic zero. Throughout the paper, a variety is an integral separated scheme of finite type over $\kk$. 

\medskip

\noindent \textbf{Acknowledgements.\ } 
We would like to thank the organisers of the `Facets of Noncommutative Geometry' conference, held at the University of Illinois Urbana-Champaign in June 2022, for the opportunity to present a talk on this work in honour of Tom Nevins' memory. Tom had a great impact on all of us, mathematically and non-mathematically, and in particular on themes related to this article. We would also like to thank the referee for valuable comments and suggestions. The first two authors were partially supported by Research Project Grant RPG-2021-149 from the Leverhulme Trust. The first author was also partially supported by EPSRC grant EP-W013053-1.

\section{Background}

\subsection{Birational geometry}
 Consider a projective morphism $f\colon X\to Y$ of normal varieties over $\kk$, where $Y$ is affine. The relative Picard group is  $\Pic(X/Y):=\Pic(X)/f^*\Pic(Y)$, and we set $\Pic(X/Y)_{\QQ}:=\Pic(X/Y)\otimes_{\ZZ} \QQ$.
 A line bundle $L\in \Pic(X/Y)_{\QQ}$ is \emph{nef} (over $Y$) if $\deg L\vert_\ell\geq 0$ for every proper curve $\ell$ in $X$, and it is \emph{semiample} (over $Y$) if $L^{m}$ is basepoint-free for some $m\geq 1$. The stable base locus of $L$ is defined to be the intersection of the base loci of the linear series $\vert L^{m}\vert$ for all $m\geq 1$, and we say that $L$ is \emph{movable} if its stable base locus is of codimension at least two in $X$. Every semiample line bundle
 is nef, but the converse is not true in general. 
 
 The \emph{nef cone} of $X$ over $Y$ is the closed convex cone $\Nef(X/Y)$ in $\Pic(X/Y)_{\QQ}$ generated by line bundles on $X$ that are nef over $Y$. The relative version of Kleiman's ampleness criterion \cite[IV,~\S4]{Kleiman66} implies that the relative \emph{ample cone} $\Amp(X/Y)$ is the interior of $\Nef(X/Y)$. The \emph{movable cone} $\Mov(X/Y)$ is the closed convex cone in $\Pic(X/Y)_{\QQ}$
 obtained as the closure of the cone generated by all movable divisor classes. Note that the nef cone is contained in the movable cone. 
 
 Let $\tau\colon X\to X_0$ be a projective, surjective morphism over $Y$ satisfying $\tau_*(\mathcal{O}_X)=\mathcal{O}_{X_0}$, so $\tau$ has connected fibres.  We say that $\tau$ is of \emph{fibre type} if $\dim X_0<\dim X$. Otherwise, $\tau$ is birational, and there are two cases: either the exceptional locus of $\tau$, denoted $\Exc(\tau)$, contains a divisor, in which case $\tau$ is a \emph{divisorial contraction}; or it does not, in which case $\tau$ is a \emph{small contraction}. 
   In the latter case, let $L$ be a line bundle on $X$ such that $L^{-1}$ is $\tau$-ample. The \emph{flip of $\tau$ with respect to $L$} is a commutative diagram
   \begin{equation}
\label{eqn:psiXprime}
\xymatrix{
 X \ar[rd]_{\tau}\ar@{-->}[rr]^{\psi} & & X^\prime \ar[ld]^{\tau^\prime} \\
  & X_0 & 
 }
\end{equation}
 where $\tau^\prime$ is a small contraction, $\psi$ is an isomorphism in codimension one, and the strict transform of $L$ along $\psi$ is $\tau^\prime$-ample. If, in addition, the canonical class $K_X$ satisfies $K_X\cdot \ell=0$ for each curve $\ell$ contracted by $\tau$, then \eqref{eqn:psiXprime} is the \emph{flop} of the curve class $\ell$ \cite[Definition~6.10]{KollarMori}.
 
 Let $L\in \Pic(X/Y)_{\QQ}$ be such that the section ring
 \[
 R(X,L):=\bigoplus_{m\geq 0} f_* L^{m}
 \]
 is a finitely generated $\mathcal{O}_Y$-algebra. Then $X(L):= \Proj_Y R(X,L)$ fits into a commutative diagram
   \begin{equation}
\label{eqn:psiDXD}
\xymatrix{
 X \ar[rd]_{f}\ar@{-->}[rr]^{\psi_L} & & X(L) \ar[ld]^{f_L} \\
  & Y & 
 }
\end{equation}
 where $\psi_L$ is regular on the complement of the stable base locus of $L$ in $X$. We do not assume in general that $f$ is birational, nor do we assume that $L$ is \emph{big}, i.e.\ $\psi_L$ need not be birational either. However, if $L$ is movable, then the rational map $\psi_L$ is an isomorphism in codimension one. If $L$ is movable and $X(L)$ is $\QQ$-factorial, then we call $X(L)$ a \emph{$\QQ$-factorial small birational model}
 of $X$ over $Y$. When $X$ and $X(L)$ are $\QQ$-factorial, we identify 
 $\Pic(X(L)/Y)_{\QQ}$ with $\Pic(X/Y)_{\QQ}$ by taking strict transform along the birational map $\psi_L$; this in turn identifies $\Mov(X(L)/Y)$ with $\Mov(X/Y)$. Let $\psi_L^*\Amp(X(L)/Y)$ and $\psi_L^*\Nef(X(L)/Y)$ denote the cones in $\Pic(X/Y)_\QQ$ obtained by taking the strict transform along $\psi_L$ of all classes on $X(L)$ that are relatively ample and nef respectively.  
 
   Given $L, L^\prime\in \Pic(X/Y)_{\QQ}$ with finitely generated section rings, we say that $L$ is \emph{Mori equivalent} to $L^\prime$ if there is an isomorphism $\varphi\colon X(L)\to X(L^\prime)$ such that the rational maps $\psi_L, \psi_{L^\prime}$ satisfy $\varphi\circ \psi_L = \psi_{L^\prime}$. 
  A \emph{Mori chamber} is a Mori equivalence class whose interior is open in $\Pic(X/Y)_{\QQ}$. These chambers are typically studied under the additional assumption that $\Pic(X/Y)_\QQ$ is isomorphic to the N\'{e}ron--Severi space $N^1(X/Y):=\Pic(X/Y)_\Q/\!\equiv$ of \emph{numerical equivalence} classes, where $L\equiv L^\prime$ if and only if $\deg(L\vert_\ell) = \deg(L^{\prime}\vert_\ell)$ for every proper curve $\ell$ in $X$. 
  
  To see how the isomorphism $\Pic(X/Y)_\QQ\cong N^1(X/Y)$ arises in the case of interest to us,  recall first the following fundamental and well-known result.
 
 \begin{prop}
  If $L\in \Pic(X/Y)$ is semi-ample over $Y$, then it is nef over $Y$. Moreover, the section ring $R(X,L)$ is a finitely generated $\mathcal{O}_Y$-algebra, and the morphism from $X$ to $\Proj_Y R(X,L)$ determined by any power of $L$ contracts a proper curve $\ell$ in $X$ if and only if $L\cdot \ell = 0$.
 \end{prop}
 \begin{proof}
Suppose that $L^m\in \Pic(X/Y)$ is a basepoint-free line bundle over $Y$. The induced morphism $h\colon X\to \vert L^m\vert\cong \mathbb{P}_Y^N$ satisfies $h^*(\mathcal{O}(1))\cong L^m$. For a proper curve $\ell $ in $X$, we have
  \[
  L\cdot \ell:=  \frac{1}{m} \deg\big(h^*(\mathcal{O}(1))\vert_\ell\big) = \frac{1}{m}\deg\big(\mathcal{O}(1)\vert_{h_*[\ell]}\big)
  \]
  where $h_*[\ell]$ is the pushforward of the curve class of $\ell$. Thus, $L$ is nef over $Y$, and   $\ell$ is contracted by $h$ if and only if $L\cdot \ell=0$. Finite generation of $R(X,L)$ is the relative version of a theorem of Zariski (see \cite[Lemma~6.11]{Ohta20}),
  and the image of $h$ is $\Proj_Y R(X,L)$.
  \end{proof}

  \begin{cor}
  \label{cor:PicN1}
If each $L\in \Pic(X/Y)$ that is nef over $Y$ is actually semiample over $Y$, then there is an isomorphism $\Pic(X/Y)_\QQ \cong N^1(X/Y)$.
  \end{cor}
 \begin{proof}
  The quotient map $\Pic(X/Y)_\QQ \to N^1(X/Y)$ is injective (see \cite[Proposition~3.2]{Ohta20}). 
  \end{proof}
 
\subsection{GIT quotients and the linearisation map} \label{ss:GIT-quotients}

Let $G$ be a reductive algebraic group acting on an affine variety $V$ with coordinate ring $\kk[V]$. Let $G^\vee$ denote the character group of $G$. For $\theta\in G^\vee$, we say that $f\in \kk[V]$ is \emph{$\theta$-semi-invariant} if $f(g.v) = \theta(g)f(v)$ for all $v\in V$ and $g\in G$, and we write $\kk[V]_\theta$ for the space of $\theta$-semi-invariant functions. A point $v\in V$ is \emph{$\theta$-semistable} if there exists $j>0$ and $f\in \kk[V]_{j\theta}$ such that $f(v)\neq 0$. The \emph{$\theta$-semistable locus} $V^\theta\subseteq V$ is the $G$-invariant, open subset of $\theta$-semistable points. A point $v\in V^\theta$ is \emph{$\theta$-stable} if the stabiliser $G_v$ is finite and the orbit $G\cdot v$ is closed in $V^\theta$. 
A character $\theta\in G^\vee$ is \emph{effective} if $V^\theta$ is non-empty, and an effective character $\theta$ is \emph{generic} if every $\theta$-semistable point of $V$ is $\theta$-stable. The $\theta$-semistable locus is unchanged if we replace $\theta$ by a positive multiple, so the definitions extend to any fractional character $\theta\in G^\vee_{\QQ}:= G^\vee\otimes_{\ZZ} \QQ$.

 For any effective $\theta\in G^\vee_{\QQ}$, the GIT quotient
\[
\Xtheta{}:= V\git_\theta \, G:= \Proj \Big(\bigoplus_{j\geq 0} \kk[V]_{j\theta}\Big)
\]
 is the categorical quotient of the $\theta$-semistable locus $V^\theta$ by the action of $G$. Note that  $\Xtheta{}$ is projective over the affine quotient
 \[
Y:= V\git_0 \, G = \Spec \kk[V]^G.
\]
If $\theta$ is generic, then $\Xtheta{}$ is the geometric quotient of $V^\theta$ by $G$.

 The set of effective fractional characters is a closed, convex cone in $G^\vee_{\QQ}$ that admits a wall-and-chamber structure as follows. Fractional characters $\theta, \theta'\in G^\vee_{\QQ}$ are \emph{GIT-equivalent} if $V^{\theta}=V^{\theta'}$.  The GIT-equivalence classes form the relative interiors of a finite collection of rational polyhedral cones in $G^\vee_{\QQ}$, and the collection of all such cones, called \emph{GIT cones}, forms a fan, called the \emph{GIT fan}, whose support is the convex cone of effective fractional characters in $G^\vee_{\QQ}$. The set of generic stability parameters $\theta$ in $G^\vee_\QQ$ decomposes as the union of (GIT) \emph{chambers}, each of which is the interior of a top-dimensional cone in the GIT fan. As shown by Ressayre~\cite{RessayreThickWall}, it can happen that the interior of a top-dimensional GIT cone is not a chamber. However, in this paper we work only with stability parameters $\theta$ lying in the closure of the union of all GIT chambers, and we reserve the phrase \emph{GIT wall} for any codimension-one face of the closure $\overline{C}$ of some GIT chamber $C$.
 The characterisation of the GIT fan via GIT-equivalence was established by Ressayre~\cite{Ressayre} (see Halic~\cite{Halic} for affine $V$), building on the earlier work of Dolgachev and Hu~\cite{DolgachevHu98}, and Thaddeus~\cite{Thaddeus96}. Those papers assume that $V$ is normal, but in fact, GIT-equivalence is unaffected by passing to the normalisation of $V$; explicitly, if $\nu\colon \widetilde{V}\to V$ is the normalisation, then ${\widetilde{V}}^\theta = \nu^{-1}(V^\theta)$ for any $\theta\in G^\vee_{\QQ}$.

 Let $C$ be a GIT chamber and fix $\theta\in C$, so $\theta$ is generic. For $\chi\in G^\vee$, consider the $G$-equivariant line bundle $\chi\otimes \mathcal{O}_{V^\theta}$ on the $\theta$-stable locus in $V$ given by equipping the trivial line bundle with the action of $G$ on each fibre given by $\chi$; explicitly, the action of $G$ on $V^\theta$ lifts to the action on $V^\theta\times \mathbb{A}^1$ such that the dual action on functions is $g\cdot (f,t) = (g\cdot f, \chi^{-1}(g) t)$. It follows that the space of sections is isomorphic to the space $\kk[V^\theta]_\chi$ of $\chi$-semi-invariant functions on $V^\theta$. By descent \cite{Nevins08}, $\chi\otimes \mathcal{O}_{V^\theta}$ descends to a line bundle on $\Xtheta{}$ if the stabiliser of each $x \in V^{\theta}$ is in the kernel of $\chi$. Since all stabilisers are finite, and there are only finitely many conjugacy classes of such stabilisers by \cite[Corollaire~3]{Luna}, there is some multiple $j \chi\in G^\vee$ of $\chi$ that descends to a line bundle on $X$ that we denote $L_{j\chi}$. We define $L_{\chi}:= \frac{1}{j}L_{j \chi}\in \Pic(\Xtheta{}/Y)_{\QQ}$.  
   
 \begin{defn}
  Let $C$ be a GIT chamber. For $\theta\in C$ and $\Xtheta{}=V\git_\theta \, G$, the \emph{linearisation map} for $C$ is the $\mathbb{Q}$-linear map 
 \begin{equation}
     \label{eqn:ellC}
 L_C\colon G^\vee_{\mathbb{Q}} \longrightarrow \Pic(\Xtheta{}/Y)_{\QQ} 
 \end{equation}
 determined by setting $L_C(\chi):= L_\chi$ for all $\chi\in G^\vee$.
  \end{defn}
  
\subsection{Variation of GIT quotient}
Let $C$ be a GIT chamber and let $\theta\in C$. In addition, let $\theta_0$ be a general point in any face of the closure $\overline{C}$. The $G$-equivariant inclusion of the $\theta$-stable locus into the $\theta_0$-semistable locus of $V$ fits into a commutative diagram of varieties
\begin{equation}
     \label{eqn:VGITdiagram}
 \begin{tikzcd}
 V^{\theta} \ar[r,hook]\ar[d,"\pi"] & V^{\theta_0}\ar[d,"\pi_0"] \\
 \Xtheta{} \ar[r,"\tau"] & \Xtheta{_0}
 \end{tikzcd}
  \end{equation}
 where $\pi_0$ is a good categorical quotient, $\pi$ is a geometric quotient and $\tau$ is a projective morphism; the morphism $\tau$ is said to be induced by \emph{variation of GIT quotient} (VGIT). The $G$-equivariant line bundle $\theta_0\otimes \mathcal{O}$ on $V^{\theta_0}$ descends to the polarising ample bundle $\mathcal{O}(1)$ on $\Xtheta{_0}$, and its restriction to $V^\theta$ descends to the line bundle $L_C(\theta_0)$ on $\Xtheta{}$. Commutativity of diagram \eqref{eqn:VGITdiagram} gives
 \begin{equation}
     \label{eqn:O1pullback}
     L_C(\theta_0) = \tau^*\big(\mathcal{O}(1)\big).
     \end{equation}

 Let $C_-$ and $C_+$ be adjacent GIT chambers separated by a wall. Let $\theta_-\in C_-$, $\theta_+\in C_+$ and let $\theta_0$ be a general point in the wall $\overline{C_-}\cap \overline{C_+}$. The morphisms $\tau_-\colon \Xtheta{_-}\rightarrow \Xtheta{_0}$ and $\tau_+\colon \Xtheta{_+}\rightarrow \Xtheta{_0}$ 
 obtained by VGIT as in \eqref{eqn:VGITdiagram} fit into a commutative diagram
 \begin{equation}
\begin{tikzcd}
\label{eqn:flop}
 \Xtheta{_-} \ar[rr,"\psi",dashed] \ar[dr,"\tau_-"'] & & \Xtheta{_+} \ar[dl,"\tau_+"] \\
& \Xtheta{_0} & 
\end{tikzcd}
\end{equation}
 of varieties over $Y=V\git_0 \, G$. Let us assume that $\Xtheta{_\pm}$ are normal. 
 
 \begin{lem}
 \label{lem:theta0stable}
 \begin{enumerate}
     \item[\one] The $\theta_0$-stable locus in $V$ is the intersection $V^{\theta_0\mathrm{-st}}:=V^{\theta_+}\cap V^{\theta_-}$.
     \item[\two] Each map in diagram \eqref{eqn:flop} is an isomorphism over the subset $V^{\theta_0\mathrm{-st}}/G\subseteq \Xtheta{_0}$.
     \item[\three] The subset $\tau_-^{-1}\big(\pi_0(V^{\theta_0}\setminus V^{\theta_0\mathrm{-st}})\big)$ is Zariski-closed in $\Xtheta{_-}$ (the same with $+$ replacing $-$).
 \end{enumerate}
 \end{lem}
 \begin{proof}
   Part \one\ follows by combining two results from Thaddeus~\cite[Proposition~1.3, Lemma~3.2]{Thaddeus96}; alternatively, the proof by Dolgachev--Hu~\cite[Proposition~3.4.7, Lemma~4.1.5]{DolgachevHu98} can be applied under our assumptions on $V$. Part \two\ follows from the description of the open set $V^{\theta_0\mathrm{-st}}$ given in part~\one. For part \three, the semistable locus for any character of $G$ is open and $G$-invariant in $V$, so $V^{\theta_0}\setminus V^{\theta_-}$ is a closed and $G$-invariant subset of $V^{\theta_0}$. The image $\pi_0(V^{\theta_0}\setminus V^{\theta_-})$ is closed because $\pi_0$ is a good quotient, so the statement follows from continuity of $\tau_-$ in the Zariski topology.
 \end{proof}

\begin{defn}
   The \emph{unstable locus} for $\tau_-$ is the subset in $X_{\theta_-}$ parametrising strictly $\theta_0$-semistable points, $\Uns(\tau_-):=\tau_-^{-1}\big(\pi_0(V^{\theta_0}\setminus V^{\theta_0\mathrm{-st}})\big)$. The locus $\Uns(\tau_+)$ is the same with $+$ replacing $-$.
\end{defn}

 \begin{rem}
 \label{rem:exceplocus}
  \begin{enumerate}
     \item \label{exceptional-unstable} Lemma~\ref{lem:theta0stable} implies that the exceptional locus of $\tau_-$ (i.e.\ the locus where it is not an isomorphism), denoted $\Exc(\tau_-)$, is a subset of the unstable locus $\Uns(\tau_-)$. This inclusion may be strict. The locus $\Exc(\tau_+)$ is the same with $+$ replacing $-$.
     \item The unstable locus is closed by Lemma~\ref{lem:theta0stable}\three. We give it the reduced scheme structure.
   \end{enumerate}
 \end{rem}

The next result generalises a result of Thaddeus~\cite[Theorem~3.3]{Thaddeus96}, though the necessary assumption on the dimension of the unstable locus is missing from that statement. For this, we take the Stein factorisations of the morphisms $\tau_\pm$ from diagram \eqref{eqn:flop} to obtain a commutative diagram
 \begin{equation}
\begin{tikzcd}
\label{eqn:flopStein}
 \Xtheta{_-} \ar[rr,"\psi",dashed] \ar[dr,"\widetilde{\tau}_-"'] & & \Xtheta{_+} \ar[dl,"\widetilde{\tau}_+"] \\
& \widetilde{X}_{\theta_0}, & 
\end{tikzcd}
\end{equation}
where the morphisms $\widetilde{\tau}_\pm$ have connected fibres. Note that the target $\widetilde{X}_{\theta_0}$ is just the normalisation of $X_{\theta_0}$ since $\tau_{\pm}$ are birational by Lemma~\ref{lem:theta0stable}\one\ and $\Xtheta{_\pm}$ are normal.  In other words, $\widetilde{\tau}_{\pm}$ are the functorial maps induced by normalisation.

\begin{prop}
\label{prop:floppingwall}
 Assume that $\Xtheta{_+}$ and $\Xtheta{_-}$  are normal, that $\Uns(\tau_+)\subseteq \Xtheta{_+}$ and $\Uns(\tau_-)\subseteq \Xtheta{_-}$ have codimension at least two, and that $\tau_+$ and $\tau_-$ contract at least one curve. Then \eqref{eqn:flopStein} is a flip with respect to the line bundle $L_{C_-}(\theta_+)$
 on $\Xtheta{_-}$, and
  \begin{equation}
     \label{eqn:LC+isLC-}
     L_{C_-}(\eta) \cong \psi^*\big(L_{C_+}(\eta)\big) \text{ for all }\eta\in G^\vee_{\mathbb{Q}}.
 \end{equation}     
  If, in addition, $\Xtheta{_-}$ is $\QQ$-factorial and $L_{C_-}$ is surjective, then $\Xtheta{_+}$ is $\QQ$-factorial.
 \end{prop} 

\begin{proof}
  The exceptional loci $\Exc(\tau_+)$ and $\Exc(\tau_-)$ are contained in the unstable loci $\Uns(\tau_+)$ and $\Uns(\tau_-)$ respectively by Remark~\ref{rem:exceplocus}(1), so our codimension assumption shows that $\psi$ is an isomorphism in codimension one. Pushforward along $\psi$ therefore identifies the class groups of $X_{\theta_\pm}$. 
 
 Combining Lemma~\ref{lem:theta0stable}\two\ with our assumption on the unstable locus implies that the complement of $V^{\theta_0\mathrm{-st}}/G$ is of codimension at least two in both $\Xtheta{_+}$ and $\Xtheta{_-}$. Since $\Xtheta{_+}$ and $\Xtheta{_-}$  are normal, line bundles on both $\Xtheta{_+}$ and $\Xtheta{_-}$ are uniquely determined, up to isomorphism, by their restriction to $V^{\theta_0\mathrm{-st}}/G$. By restricting all three maps from \eqref{eqn:flop} to the isomorphisms over this locus, we see that both $L_{C_+}(\eta)$ on $\Xtheta{_+}$ and $L_{C_-}(\eta)$ on $\Xtheta{_-}$ are obtained by descent from $\eta\otimes \mathcal{O}$ on $V^{\theta_0\mathrm{-st}}$. Therefore isomorphism \eqref{eqn:LC+isLC-} holds. In particular, the strict transform of $L_{C_-}(\theta_+)$ along $\psi$ is the line bundle $L_{C_+}(\theta_+)$ on $\Xtheta{_+}$. 
  
 The polarising ample bundle  $L_{C_+}(\theta_+)$ on $\Xtheta{_+}$ is $\widetilde{\tau}_+$-ample, so to prove that \eqref{eqn:flopStein} is a flip, we need only show that $L_{C_-}(\theta_+)^{-1}$ is $\widetilde{\tau}_-$-ample. For this, the ample bundle $L_0:=\mathcal{O}(1)$ on $\widetilde{X}_{\theta_0}$ satisfies $L_{C_-}(\theta_0) = (\widetilde{\tau}_-)^*(L_0)$ by \eqref{eqn:O1pullback}.
 By choosing alternative characters $\theta_+\in C_+$ and $\theta_-\in C_-$ if necessary, we may assume that $\theta_0 = \frac{1}{2}(\theta_++\theta_-)$, in which case  
\[
L_{C_-}(\theta_+)\otimes L_{C_-}(\theta_-) = L_{C_-}(\theta_+ + \theta_-) = L_{C_-}(2\theta_0) = (\widetilde{\tau}_-)^*(L_0)^{2}.
\]
The set of curve classes contracted by $\widetilde{\tau}_-$ is non-empty by assumption. The line bundles $L_{C_-}(\theta_-)$ and $\widetilde{\tau}_-^*(L_0)$ have positive and zero degree respectively on all such curves. It follows that
\[
L_{C_-}(\theta_+)^{-1} = L_{C_-}(\theta_-)\otimes  \tau_-^*(L_0)^{-2}
\]
has positive degree on all such curves, so it is $\widetilde{\tau}_-$-ample as required.

 For the final statement, let $D_+$ be a Weil divisor on $\Xtheta{_+}$. Then $(\psi^{-1})_*D_+$ is a Weil divisor on $\Xtheta{_-}$, and since $\Xtheta{_-}$ is $\QQ$-factorial, the divisor $m(\psi^{-1})_*D_+$ is Cartier for some $m>0$. Since $L_{C_-}$ is surjective, there exists $\eta\in G^\vee_\QQ$ such that $L_{C_-}(\eta) \cong \mathcal{O}_{\Xtheta{_-}}(m(\psi^{-1})_*D_+)$. Now \eqref{eqn:LC+isLC-} gives
 \[
  \mathcal{O}_{\Xtheta{_+}}(mD_+) \cong (\psi^{-1})^*\mathcal{O}_{\Xtheta{_-}}(m(\psi^{-1})_*D_+) \cong (\psi^{-1})^* L_{C_-}(\eta) \cong L_{C_+}(\eta)
 \]
 which lies in $\Pic(\Xtheta{_+}/Y)$, so $mD_+$ is Cartier.
 \end{proof}

\section{Reconstructing relative Mori Dream Spaces by GIT}

\subsection{GIT regions}
As before, let $G$ denote a reductive algebraic group acting on an affine variety $V$. For $\theta\in G^\vee_{\mathbb{Q}}$, write $\Xtheta{}:=V\git_\theta \, G$, and let $f\colon \Xtheta{}\to Y:= X_0$ denote the projective morphism obtained by VGIT.

 To formulate our key condition, let $C_-$ and $C_+$ be GIT chambers separated by a wall $\overline{C_-}\cap \overline{C_+}$. We delete this separating wall if and only if the morphisms $\tau_-$ and $\tau_+$ from diagram \eqref{eqn:flop} are both small. The result is an a priori coarser wall-and-chamber decomposition of the GIT fan. 

 \begin{defn}
 \label{def:GITregion}
  A \emph{GIT region} in $G^\vee_\QQ$ is any top-dimensional cone of the coarse fan defined above. By construction, every GIT region that contains a chamber is the union of the closures of a collection of GIT chambers. For any chamber $C$, let $R_C$ denote the unique GIT region containing $C$.
 \end{defn}
 
\begin{example}
\label{exa:F1}
 The Cox construction of the first Hirzebruch surface $\mathbb{F}_1:=\mathbb{P}_{\mathbb{P}^1}(\mathcal{O}\oplus \mathcal{O}(1))$ passes via the action of the torus $G=(\kk^\times)^2$ on $\mathbb{A}^4$ with weights $\begin{bsmallmatrix}1&-1&1& 0\\0&1&0&1\end{bsmallmatrix}$. There are two GIT chambers
 \[
 C_-:= \Amp(\mathbb{F}_1) = \big\{\alpha \begin{bsmallmatrix}1\\0\end{bsmallmatrix} +\beta \begin{bsmallmatrix}0\\1\end{bsmallmatrix} \mid \alpha, \beta>0\big\}
 \quad\text{and}\quad 
 C_+:= \big\{\alpha \begin{bsmallmatrix}-1\\1\end{bsmallmatrix} +\beta \begin{bsmallmatrix}0\\1\end{bsmallmatrix} \mid \alpha, \beta>0\big\}
 \]
 For $\theta_-\in C_-$ and $\theta_+\in C_+$, we have $\Xtheta{_-}\cong \mathbb{F}_1$ and $\Xtheta{_+}\cong \mathbb{P}^2$. The VGIT morphism $\tau_-\colon \mathbb{F}_1\to \mathbb{P}^2$ contracts the $(-1)$-curve, whereas $\tau_+$ is an isomorphism. Thus, $R_{C_-} = \overline{C_-}$ and $R_{C_+} = \overline{C_+}$. 
 \end{example} 
 
 The linearisation map $L_{C_-}$ in Example~\ref{exa:F1} is an isomorphism that identifies $R_{C_-}$ with the movable (in fact, the nef) cone of $\Xtheta{_-}$ for $\theta_-\in C_-$. Example~\ref{exa:threefoldflop} below 
 illustrates that even when $L_{C_-}$ is an isomorphism, it need not identify $R_{C_-}$ with $\Mov(\Xtheta{_-})$ for $\theta\in C_-$.

 \subsection{The new GIT condition and two key examples}
 Define a GIT wall to be a \emph{flipping wall} if it satisfies the assumptions of Proposition~\ref{prop:floppingwall}, i.e.\ the wall separates two GIT chambers $C_\pm$ such that, in the notation of diagram \eqref{eqn:flop}, both $\Xtheta{_-}$ and $\Xtheta{_+}$ are normal, the unstable loci  $\Uns(\tau_+)\subseteq \Xtheta{_+}$ and $\Uns(\tau_-)\subseteq \Xtheta{_-}$ have codimension at least two, and the morphisms $\tau_+$ and $\tau_-$ both contract at least one curve. For any wall in the boundary of $R_C$, let $C_-$ denote the unique chamber in $R_C$ that contains the wall in its closure. Let $\theta_-\in C_-$ and let $\theta_0$ be general in the wall.  We say that the wall, when approached from the chamber $C_-$ in $R_C$, is \emph{small, divisorial} or \emph{of fibre type} if the
 induced morphism $\tau_-\colon \Xtheta{_-} \to \Xtheta{_0}$ is of the same type.

\begin{cond}
\label{cond:GIT} 
There exists a GIT chamber $C$ such that:
\begin{enumerate}
    \item[\emph{(1)}] for $\theta\in C$, the GIT quotient $X:=\Xtheta{}$ is a $\QQ$-factorial normal variety and the linearisation map 
    \[
    L_{C} \colon G^{\vee}_\QQ\longrightarrow \Pic(X/Y)_\Q
    \]
    is an isomorphism of rational vector spaces;
    \item[\emph{(2)}] each wall in the interior of the GIT region $R_C$ containing $C$ is a flipping wall; and
    \item[\emph{(3)}] each boundary wall of $R_C$ is either divisorial or of fibre type.
\end{enumerate}
\end{cond}

 \begin{remark}
 \label{rem:keyExamples}
 Condition~\ref{cond:GIT} is required for the statement and proof of Theorem~\ref{thm:movable}. We present several examples to shed light on the three different parts of this assumption as follows:
 \begin{itemize}
     \item Condition~\ref{cond:GIT} holds for the chamber $C_-$ in Example~\ref{exa:F1}, while Condition~\ref{cond:GIT}(1) fails for $C_+$. Many similar examples are described in Example~\ref{exa:MDS} below.
     \item If Condition~\ref{cond:GIT}(1) holds, it can happen that (2) fails to hold; see Example~\ref{exa:toricQuiverGIT} below.
     \item If Condition~\ref{cond:GIT}(1) and (2) both hold, it can happen that (3) fails to hold; see Example~\ref{exa:threefoldflop}.
 \end{itemize}
 Note in addition that Condition~\ref{cond:GIT}(2) implies that $\Xtheta{}$ is normal for every generic $\theta\in R_C$. 
 \end{remark}

\begin{example}[\textbf{Mori Dream Spaces via the Cox ring}] 
\label{exa:MDS}
 Generalising Example~\ref{exa:F1}, let $X$ be any Mori Dream Space in the sense of Hu and Keel~\cite{HuKeel00}. That is, $X$ is a $\QQ$-factorial normal projective variety with $\Pic(X)_\QQ\cong N^1(X)$, such that the Cox ring of $X$, denoted $\Cox(X)$, is a finitely generated $\kk$-algebra. For simplicity, assume that $\Pic(X)$ is free. The $\Pic(X)$-grading of $\Cox(X)$ defines an action of the algebraic torus $G:= \Hom(\Pic(X),\kk^\times)$ on the affine variety $V:= \Spec \Cox(X)$.
 \begin{enumerate}
 \item[\one] For the chamber $C=\Amp(X)$, the linearisation map $L_C \colon G^\vee_\QQ\to \Pic(X)_\QQ$ is an isomorphism by \cite[Proof of Proposition~2.11]{HuKeel00}, so Condition~\ref{cond:GIT}(1) holds, whilst \cite[Proposition~1.11]{HuKeel00} shows that conditions (2) and (3) also hold.  
 \item[\two] For any chamber $C^\prime$ that does not lie in $\Mov(X)$,  the kernel of $L_{C^\prime}$ has dimension at least one because the rank of $\Pic(\Xtheta{^\prime})_\QQ$ for $\theta^\prime\in C^\prime$ drops by one as we cross each boundary wall of the movable cone. In particular, Condition~\ref{cond:GIT} fails for $C^\prime$. 
 \end{enumerate}
 Analogous statements hold for the action induced by the $\Pic(X/Y)$-grading on the Cox ring for any relative Mori Dream Space $X\to Y$; this requires a choice of line bundles on $X$ that provide a basis for $\Pic(X/Y)_{\QQ}$ (see Grab~\cite{Grab19} or Ohta~\cite{Ohta20} for details).
\end{example}
 
\begin{example}[\textbf{A local del Pezzo by quiver GIT}]
\label{exa:toricQuiverGIT}
Let $Z$ be the two-point blow-up of $\mathbb{P}^2$. The total space $X:= \text{tot}(\omega_Z)$ of the canonical bundle on $Z$ is smooth with trivial canonical class, the anticanonical ring $R:= \bigoplus_{k\geq 0} H^0(Z,\omega_Z^{-\otimes k})$ is Gorenstein \cite[Example~5.1.13]{GotoWatanabe},  and the morphism 
\[
f\colon X=\text{tot}(\omega_Z) \longrightarrow Y:= \Spec R
\]
 that contracts the zero section is a projective crepant resolution. In fact, $f$ is a morphism of toric varieties: for the lattice $M=\ZZ^3$, we have that $R\cong \CC[\sigma^\vee\cap M]$, where $\sigma\subseteq N\otimes_{\ZZ} \QQ$ is the strongly convex rational polyhedral cone obtained as the cone over the  pentagon in Figure~\ref{fig:totwZ}(a); the basic triangulation of the pentagon that determines the fan $\Sigma$ of $X$ is also shown in Figure~\ref{fig:totwZ}(a).
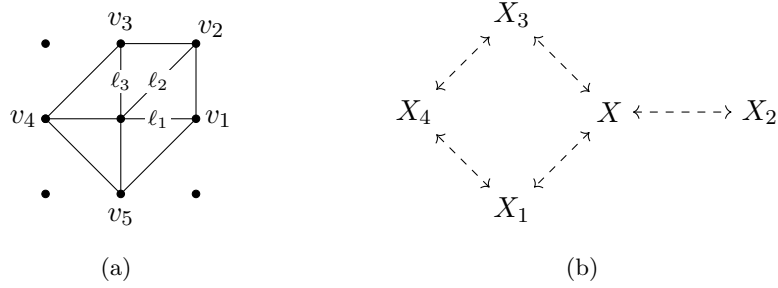
\begin{figure}[!ht]
   \centering
      \subfigure[]{
\begin{tikzpicture}[baseline={(0,-1.5)}] 
\draw (1,0) -- (1,1) -- (0,1) -- (-1,0) -- (0,-1) -- (1,0); 
\draw (0,0) node[circle,draw,fill=black,minimum size=2pt,inner sep=1pt] {{}};
\draw (1,0) node[circle,draw,fill=black,minimum size=2pt,inner sep=1pt] {{}};
\draw (-1,0) node[circle,draw,fill=black,minimum size=2pt,inner sep=1pt] {{}};
\draw (0,1) node[circle,draw,fill=black,minimum size=2pt,inner sep=1pt] {{}};
\draw (1,1) node[circle,draw,fill=black,minimum size=2pt,inner sep=1pt] {{}};
\draw (-1,1) node[circle,draw,fill=black,minimum size=2pt,inner sep=1pt] {{}};
\draw (-1,0) node[circle,draw,fill=black,minimum size=2pt,inner sep=1pt] {{}};
\draw (0,-1) node[circle,draw,fill=black,minimum size=2pt,inner sep=1pt] {{}};
\draw (1,-1) node[circle,draw,fill=black,minimum size=2pt,inner sep=1pt] {{}};
\draw (-1,-1) node[circle,draw,fill=black,minimum size=2pt,inner sep=1pt] {{}};
\draw (-1,0) -- (1,0);
\draw  (0,-1) -- (0,1);
\draw  (0,0) -- (1,1);
\draw (1.3,0) node {$v_1$};
\draw (1.2,1.2) node {$v_2$};
\draw (0,1.3) node {$v_3$};
\draw (-1.3,0) node {$v_4$};
\draw (0,-1.3) node {$v_5$};
\draw (0.5,0)  node [fill=white,inner sep=1pt] {${\scriptstyle \ell_1}$};
\draw (0.5,0.5)  node [fill=white,inner sep=1pt] {${\scriptstyle \ell_2}$};
\draw (0,0.5)  node [fill=white,inner sep=1pt] {${\scriptstyle \ell_3}$};
\end{tikzpicture}
 \label{fig:totwZfan}
 }
      \qquad 
      \qquad
      \subfigure[]{
\begin{tikzpicture}
\node (A) at (1.3,1.3) {$X$};
\node (B) at (0,0) {$X_1$};
\node (C) at (3.3,1.3) {$X_2$};
\node (D) at (0,2.6) {$X_3$};
\node (E) at (-1.3,1.3) {$X_4$};
\path (A) edge[dashed, <->] (C);
\path (A) edge[dashed, <->] (B);
\path (A) edge[dashed, <->] (D);
\path (D) edge[dashed, <->] (E);
\path (B) edge[dashed, <->] (E);
\end{tikzpicture} 
\label{fig:totwZflops}
}
          \caption{(a) slice of the fan defining $X$; (b) flops linking crepant resolutions}
          \label{fig:totwZ}
  \end{figure} 
  The cones $\ell_1, \ell_2, \ell_3\in \Sigma$ each determine a $(-1,-1)$-curve in $X$, and for $1\leq i\leq 3$, flopping the curve defined by $\ell_i$ produces a projective crepant resolution $X_i\to Y$. A curve in each of $X_1$ and $X_3$ can be flopped to produce the projective crepant resolution $X_4\to Y$ as shown in Figure~\ref{fig:totwZflops}.

The morphism $f\colon X\to Y$ is a relative Mori Dream Space, but our interest here lies with a GIT quotient construction that differs from the construction via the Cox ring of $X$ as in Example~\ref{exa:MDS}. For this, list the torus-invariant prime Weil divisors  $D_1, \dots, D_5$ on $Y$, one for each lattice point $v_1, \dots, v_5$ on the boundary of the pentagon as above. Define four reflexive sheaves of rank one on $Y$, namely $E_0:= \mathcal{O}_Y$, $E_1:=\mathcal{O}_Y(D_1)$, $E_2:=\mathcal{O}_Y(D_3)$ and $E_3:= \mathcal{O}_Y(D_1+D_5)$, and set $\mathscr{E}:=\{E_0, E_1, E_2, E_3\}$.  Following \cite[Definition~2.2]{CV12}, the \emph{quiver of sections} of  $\mathscr{E}$, denoted $Q$, is shown in Figure~\ref{fig:totwZquiver}: the vertex set corresponds to the collection $\mathscr{E}$; and each arrow is labelled by a Weil divisor where, for example, the label $12$ is shorthand for the divisor $D_1+D_2$.  The algebra 
 $\End_{R}(\bigoplus_{0\leq i\leq 3}E_i)$ can be presented as the quotient of the path algebra of $Q$ by a two-sided ideal of relations determined by the labelling of arrows by divisors \cite[Lemma~2.5]{CV12}. 
  
   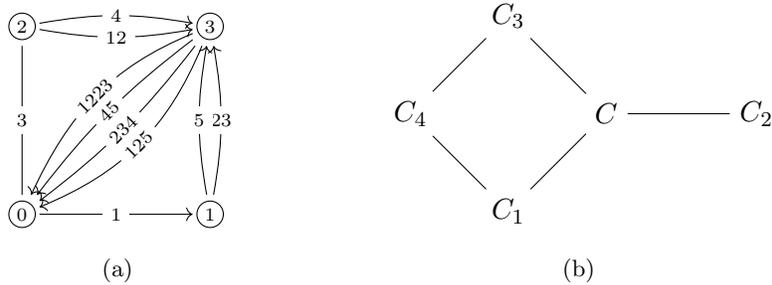
\begin{figure}[!ht]
   \centering
      \subfigure[]{
\begin{tikzpicture}
\node (A) at (0,0) {$0$};
\node (B) at (2.5,0) {$1$};
\node (C) at (0,2.5) {$2$};
\node (D) at (2.5,2.5) {$3$};
\draw (A) node[circle,draw,fill=white,minimum size=10pt,inner sep=1pt] {{\tiny 0}};
\draw (B) node[circle,draw,fill=white,minimum size=10pt,inner sep=1pt] {{\tiny 1}};
\draw (C) node[circle,draw,fill=white,minimum size=10pt,inner sep=1pt] {{\tiny 2}};
\draw (D) node[circle,draw,fill=white,minimum size=10pt,inner sep=1pt] {{\tiny 3}};
\draw[->] (A) to node[fill=white,inner sep=3pt] {{\tiny 1}} (B);
\draw[->] (A) to node[fill=white,inner sep=3pt] {{\tiny 3}}(C);
\draw[->,out=80, in=280] (B) to node[fill=white,inner sep=3pt] {{\tiny 23}} (D);
\draw[->,out=100, in=260] (B) to node[fill=white,inner sep=3pt] {{\tiny 5}} (D);
\draw[->,out=10, in=170] (C) to node[fill=white,inner sep=3pt] {{\tiny 4}} (D);
\draw[->,out=350, in=190] (C) to node[fill=white,inner sep=3pt] {{\tiny 12}} (D);
\draw[->,out=202.5, in=65] (D) to node[fill=white,rotate=45,inner sep=3pt] {{\tiny 1223}} (A);
\draw[->,out=217.5, in=52.5] (D) to node[fill=white,rotate=45,inner sep=3pt] {{\tiny 45}} (A);
\draw[->,out=232.5, in=37.5] (D) to node[fill=white,rotate=45,inner sep=3pt] {{\tiny 234}} (A);
\draw[->,out=247.5, in=22.5] (D) to node[fill=white,rotate=45,inner sep=3pt] {{\tiny 125}} (A);
\end{tikzpicture} 
\label{fig:totwZquiver}
 }
      \qquad 
      \qquad
      \subfigure[]{
\begin{tikzpicture}
\node (A) at (1.3,1.3) {$C$};
\node (B) at (0,0) {$C_1$};
\node (C) at (3.3,1.3) {$C_2$};
\node (D) at (0,2.6) {$C_3$};
\node (E) at (-1.3,1.3) {$C_4$};
\path (A) edge (C);
\path (A) edge (B);
\path (A) edge (D);
\path (D) edge (E);
\path (B) edge (E);
\end{tikzpicture} 
\label{fig:totwZchambers}
}
          \caption{(a) quiver of sections $Q$ on $Y$; (b) graph indicating chambers that lie adjacent}
          \label{fig:totwZquivermoduli}
  \end{figure}
  
   To reconstruct the morphism $f$ by quiver GIT, let $\mathbb{N}(Q)$ denote the semigroup generated by the columns of the matrix
   \[
   \left[\begin{smallmatrix}
   -1 & -1 &  0 &  0 &  0 &  0 &  1 &  1 &  1 &  1 \\
    1 &  0 & -1 & -1 &  0 &  0 &  0 &  0 &  0 &  0 \\
    0 & 1 &  0 &  0 & -1 & -1 &  0 &  0 &  0 &  0 \\
    0 &  0 &  1 &  1 &  1 &  1 & -1 & -1 & -1 & -1 \\
    1 &  0 &  0 &  0 &  0 &  1 &  1 &  0 &  0 &  1 \\
    0 &  0 &  0 &  1 &  0 &  1 &  2 &  0 &  1 &  1 \\
    0 &  1 &  0 &  1 &  0 &  0 &  1 &  0 &  1 &  0 \\
    0 &  0 &  0 &  0 &  1 &  0 &  0 &  1 &  1 &  0 \\
    0 &  0 &  1 &  0 &  0 &  0 &  0 &  1 &  0 &  1 
   \end{smallmatrix}\right];
   \]
 each column corresponds to an arrow, where the top four entries record the head and tail, while the bottom five entries record the labelling divisor. Consider the toric variety $V:= \Spec \kk[ \mathbb{N}(Q)]$ defined by the semigroup algebra of $\mathbb{N}(Q)$. The algebraic torus $G:=\Spec \kk[\Wt(Q)]$ with character lattice $\Wt(Q):=\{\theta\in \ZZ^4 \mid \sum_i \theta_i=0\}$ acts on $V$, where the weights of the action are recorded in the top four rows of the matrix. 
 Following \cite[Proposition~2.14]{CV12}, the affine quotient satisfies $V\git_0 T \cong Y$, while for each $\theta\in \Wt(Q)$, variation of GIT gives a projective, birational, toric morphism $f_\theta\colon \Xtheta{}:= V\git_\theta G\to Y$. The GIT chamber
 \[
 C:=\big\{ \theta\in G^\vee_\QQ \mid \theta_1 <0;\; \theta_2<0,\; \theta_1+\theta_2+\theta_3> 0\big\}
 \]
 gives $X\cong \Xtheta{}$ and $f=f_\theta$ for $\theta\in C$. The linearisation map $L_{C}$ is an isomorphism that identifies $C$ with the ample cone of $X$,  so Criterion~\ref{cond:GIT}(1) is satisfied.
 
 The defining inequalities of the chambers $C_1, C_2, C_3$ that lie adjacent to $C$ are shown below, together with the unique chamber $C_4$ that lies adjacent to both $C_1$ and $C_3$ as in Figure~\ref{fig:totwZchambers}:
\begin{center}
\begin{tabular}{ c c c c r }
Chamber & &  \multicolumn{3}{c}{Defining inequalities} \\
 \hline
$C_1$ &  $\quad$ &  $\theta_1 <0$ \; & \: $\theta_2>0$ & $\theta_1+\theta_3> 0$ \\ 
$C_2$ &  $\quad$ &  $\theta_1 <0$ \: & \: $\theta_2<0$ & \; $\theta_1+\theta_2+\theta_3 < 0$ \\
$C_3$ &  $\quad$ &  $\theta_1 >0$ \; & \; $\theta_2<0$ & $\theta_2+\theta_3> 0$ \\
$C_4$ &  $\quad$ &  $\theta_1 >0$ \; & \; $\theta_2>0$ & $\theta_3> 0$ \\
 & & \\
\end{tabular}
\end{center}
 Crossing the wall from $C$ to $C_1$ induces the flop $X\dashrightarrow X_1$, and symmetrically, crossing from $C$ to $C_3$ induces the flop $X\dashrightarrow X_3$. In addition, for $i\in \{1,3\}$, one can cross a wall from $C_i$ to $C_4$ to induce the flop $X_i\dashrightarrow X_4$.  However,  crossing the wall separating $C_-:=C$ from $C_+:= C_2$ does not induce the flop $X\dashrightarrow X_2$ as one might expect after comparing Figures~\ref{fig:totwZflops} and \ref{fig:totwZchambers}. Rather, the morphism $\tau_-$ from diagram \eqref{eqn:flop} contracts the curve in $X$ determined by $\ell_2$, whereas $\tau_+$ is an isomorphism. In particular, $V\git_{\theta_2}G$ for $\theta_2\in C_2$ is not $\QQ$-factorial, and $X_2\not\cong V\git_{\theta_2}G$. Therefore the wall separating $C$ from $C_2$ is not a flipping wall despite being an internal wall of $R_C$, so Condition~\ref{cond:GIT}(2) fails even though Condition~\ref{cond:GIT}(1) holds.
 \end{example}

   \begin{remark}
   Example~\ref{exa:toricQuiverGIT} is obtained by modifying that from Ishii--Ueda~\cite[Example~12.6]{IU16}, where for the additional reflexive sheaf $E_4:=\mathcal{O}_Y(2D_1+2D_5)$, a collection of GIT quotients associated to a quiver with relations defining the algebra $\End_{R}(\bigoplus_{0\leq i\leq 4}E_i)$ are studied.  The linearisation map is not injective for any chamber in that case. In Example~\ref{exa:toricQuiverGIT}, we omitted the summand $E_4$ which reduces the dimension of the space of stability conditions by one, making $L_{C}$ an isomorphism.
   \end{remark}

 \begin{example}
 \label{exa:threefoldflop}
 Consider the action of the algebraic torus $(\kk^\times)^2$ on $\mathbb{A}^5$ with weights 
 $
 \begin{bsmallmatrix}
 0 & 0 & 1 & 1 & 1 \\
 1&1&0&-1&-1
 \end{bsmallmatrix}
 $. 
 For the character $L_-:= \begin{bsmallmatrix}1\\1\end{bsmallmatrix}$, the Cox construction gives $X_-=\mathbb{A}^5\git_{L_-} (\kk^\times)^2$, where we identify the character lattice of $(\kk^\times)^2$ with $\Pic(X_-)$. For $L_+:= \begin{bsmallmatrix}2\\-1\end{bsmallmatrix}$, variation of GIT quotient determines a flop $\psi\colon X_-\dashrightarrow X_+:=\mathbb{A}^5\git_{L_+} (\kk^\times)^2$ of smooth projective toric threefolds. The ample cones of $X_-$ and $X_+$ satisfy 
 \[
 \Amp(X_-) = \big\{\alpha \begin{bsmallmatrix}1\\0\end{bsmallmatrix} +\beta \begin{bsmallmatrix}0\\1\end{bsmallmatrix} \mid \alpha, \beta>0\big\} \quad\text{and}\quad \psi^*\Amp(X_+) = \big\{\alpha \begin{bsmallmatrix}1\\0\end{bsmallmatrix} +\beta \begin{bsmallmatrix}1\\-1\end{bsmallmatrix} \mid \alpha, \beta>0\big\}.
 \]
 
Our interest lies with an alternative GIT construction introduced in~\cite{CrawSmith08}. For this, consider the globally generated line bundles $L_0:= \mathcal{O}_{X_-} = \begin{bsmallmatrix}0\\0\end{bsmallmatrix}, \; L_1:= \begin{bsmallmatrix}2\\0\end{bsmallmatrix}, \; L_2:= \begin{bsmallmatrix}0\\1\end{bsmallmatrix}\in \Pic(X_-)$, and set $\mathcal{L}:=\{L_0, L_1, L_2\}$. Let $\mathbb{N}(\mathcal{L})$ denote the semigroup generated by the columns of the matrix
   \[
   \left[\begin{smallmatrix}
   -1 & -1 & -1 &  0 &  0 &  0 &  0 &  0 &  0 &  0 &  0 \\
    1 &  0 &  0 &  1 &  1 &  1 &  1 &  1 &  1 &  1 &  1 \\
    0 &  1 &  1 & -1 & -1 & -1 & -1 & -1 & -1 & -1 & -1 \\
    0 &  0 &  1 &  1 &  1 &  1 &  0 &  0 &  0 &  0 &  0 \\
    0 &  1 &  0 &  0 &  0 &  0 &  1 &  1 &  1 &  0 &  0 \\
    2 &  0 &  0 &  0 &  0 &  0 &  0 &  0 &  0 &  1 &  1 \\
    0 &  0 &  0 &  2 &  1 &  0 &  2 &  1 &  0 &  1 &  0 \\
    0 &  0 &  0 &  0 &  1 &  2 &  0 &  1 &  2 &  0 &  1 
   \end{smallmatrix}\right],
   \]
 where for $0\leq i,j\leq 2$, each column with $-1$ in row $i$ and $+1$ in row $j$ corresponds to a given torus-invariant divisor defining a section of $L_j\otimes L_i^{-1}$; see \cite[Section~3]{CrawSmith08}. The semigroup algebra of $\mathbb{N}(\mathcal{L})$ defines the toric variety $V:= \Spec \kk[ \mathbb{N}(\mathcal{L})]$, and the top three rows of the above matrix encode the weights of an action on $V$ by the algebraic torus $G$ of rank two 
 whose character lattice is $G^\vee=\{(\theta_i)\in \ZZ^3 \mid \theta_0+\theta_1+\theta_2=0\}$.  There are two GIT chambers 
  \[
 C_-:=\{\theta\in G^\vee_{\mathbb{Q}} \mid \theta_1>0, \theta_2>0\}
 \quad\text{and}\quad 
 C_+:= \{\theta\in G^\vee_{\mathbb{Q}} \mid \theta_1+\theta_2>0, \theta_2<0\}.
 \]
 Observe that the GIT quotient $X_\eta$ is non-empty if and only if $\eta\in \overline{C_-}\cup \overline{C_+}$.
 
 We claim that 
 the GIT regions satisfy 
 $R_{C_-}= \overline{C_-}\cup \overline{C_+}=R_{C_+}$. To see this, fix $\theta_-:=(-2,1,1)\in C_-$ and $\theta_+:=(-1,2,-1)\in C_+$. Applying \cite[Corollary~4.10, Theorem~4.15]{CrawSmith08} shows that $X_-\cong \Xtheta{_-}$, and moreover, that $L_{C_-}$ identifies $G^\vee$ with the index 2 sublattice of $\Pic(X_-)$ spanned by $L_1$ and $L_2$. It follows that $L_{C_-}$ is an isomorphism of rational vector spaces that identifies $C_-$ with $\Amp(X_-)$. More generally, for $\eta\in \overline{C_-}\cup \overline{C_+}$, the columns of the matrix are chosen to ensure that the $\eta$-graded piece $\kk[V]_\eta$ of the coordinate ring of $V$ is isomorphic to the vector space $H^0(X_-,L_{C_-}(\eta))$, so $X_\eta$  is isomorphic to $\Proj$ of the section ring of $L_{C_-}(\eta)$. In particular, $\Xtheta{_+}\cong X_+$ and the rational map $\Xtheta{_-}\dashrightarrow \Xtheta{_+}$ induced by crossing the wall separating $C_-$ and $C_+$ is the flop $\psi\colon X_-\dashrightarrow X_+$. Thus, both VGIT morphisms  $\tau_\pm$ are small, so the claim follows.
 
 Note, however, that while $L_{C_-}$ is an isomorphism, it fails to identify $C_+$ with $\Amp(X_+)$, so $L_{C_-}$ does not identify $R_{C_-}$ with $\Mov(\Xtheta{_-})$. In fact, the cone $C_+$ is generated by the vectors $(-1,1,0)$ and $(0,1,-1)$, so the cone $L_{C_-}(C_+)$ is generated by $L_{C_-}((-1,1,0)) = L_0^{-1}\otimes L_1\otimes L_2^0\cong L_1=\begin{bsmallmatrix}2\\0\end{bsmallmatrix}$ and $L_{C_-}((0,1,-1)) = L_0^0\otimes L_1\otimes L_2^{-1} = \begin{bsmallmatrix}2\\-1\end{bsmallmatrix}$, giving
 \[
 L_{C_-}(C_+)=\big\{L_1^{\alpha}\otimes (L_1\otimes L_2^{-1})^{\beta} \mid \alpha, \beta>0\big\} = \big\{\alpha \begin{bsmallmatrix}2\\0\end{bsmallmatrix} + \beta \begin{bsmallmatrix}2\\-1\end{bsmallmatrix} \mid \alpha, \beta >0\big\}.
 \]
 It follows that $L_{C_-}$ identifies the character $\eta=(0,1,-1)$ in the boundary of $R_{C_-}$ with the ample bundle $\begin{bsmallmatrix}2\\-1\end{bsmallmatrix}$ on $X_+$, so the induced VGIT morphism $\Xtheta{_+}\to X_{\eta}$ is an isomorphism rather than a divisorial or fibre type contraction.
 \end{example}
 
\subsection{The main result}
We now establish the main geometric consequences of Condition~\ref{cond:GIT}. 
For $C$ a GIT chamber, $\theta \in C$, and $\eta \in \overline{C}$, let $X_\theta \to \widetilde{X}_\eta$ be the Stein factorisation of the VGIT morphism $X_\theta \to X_\eta$. 

\begin{lem}
\label{lem:Assumptiongivesnefissemiample}
Let $C$ be a chamber satisfying Condition~\ref{cond:GIT}, and set $X=\Xtheta{}$ for $\theta\in C$. Then:
\begin{enumerate}
\item[\one] the isomorphism $L_{C}$ identifies $C$ and $\overline{C}$ with $\Amp(X/Y)$ and $\Nef(X/Y)$ respectively;
\item[\two] each $L\in \Pic(X/Y)$ that is nef over $Y$ is also semiample over $Y$; 
\item[\three] for $\eta\in \overline{C}$,
the section ring $R(X,L_C(\eta))$ is a finitely generated $\mathcal{O}_Y$-algebra, and moreover, $\widetilde{X}_\eta$ is isomorphic to  
$\Proj R(X,L_C(\eta))$; and
\item[\four] $\Pic(X/Y)_\QQ\cong N^1(X/Y)$.
\end{enumerate}

\end{lem}
\begin{proof} 
 For part \one, the line bundle $L_{C}(\theta)$ is ample, so the isomorphism $L_{C}$ from Condition~\ref{cond:GIT}(1) identifies $C$ with the interior of a top-dimensional polyhedral cone in $\Nef(X/Y)$. Let $\theta_0$ be general in a wall of $\overline{C}$. If this wall passes through the interior of $R_C$, then the morphism $\tau\colon X\to \Xtheta{_0}$ induced by VGIT contracts at least one curve by Condition~\ref{cond:GIT}(2). Otherwise, it's a boundary wall of $R_C$, in which case  $\tau\colon X\to \Xtheta{_0}$ is of fibre type or it is a divisorial contraction. Thus, $L_{C}$ sends the boundary of $\overline{C}$ into the boundary of $\Nef(X/Y)$, so in fact $L_{C}$ identifies $C$ with $\Amp(X/Y)$ and $\overline{C}$ with $\Nef(X/Y)$.
 
 For part \two, any class in the interior of $\Nef(X/Y)$ is ample and hence semiample. For $L$ in the boundary of $\Nef(X/Y)$, it follows that $\eta:=L_{C}^{-1}(L)$ lies in the boundary of $C$. After multiplying by some $m>0$ if necessary, \eqref{eqn:O1pullback} shows that $L_C(m\eta) = \tau^*(\mc{O}(1))$ for the morphism $\tau\colon X\to X_{m\eta}$ over $Y$ induced by VGIT. Thus, $L^m=\tau^*\big(\mc{O}(1)\big)$ is basepoint-free, so $L$ is semiample over $Y$.
 
 For part \three, let $\eta\in \overline{C}$. Parts \one\ and \two\ imply that $L_C(\eta)$ is semiample, so the section ring $R(X,L_C(\eta))$ is finitely generated by a theorem of Zariski~\cite[Example~2.1.30]{LazarsfeldI}, and hence the model $X(L_C(\eta))$ is well-defined. Since $\eta$ lies in a face of $\overline{C}$,  the VGIT morphism $\tau\colon X\to X_\eta$ satisfies $L_C(\eta)\cong\tau^*(\mathcal{O}_{X_\eta}(1))$ by \eqref{eqn:VGITdiagram}. The Stein factorisation $\widetilde{\tau}\colon X_\theta \to \widetilde{X}_\eta$ then gives $L_C(\eta)\cong\widetilde{\tau}^*(\mathcal{O}(1))$, where $\mathcal{O}(1)$ is ample on $\widetilde{X}_\eta$. Since $\widetilde{\tau}_*(\mathcal{O}_X)\cong \mathcal{O}_{\widetilde{X}}$, we have for each $m\geq 0$ that
 \begin{equation}
     \label{eqn:normalisation}
 f_*L_C(\eta)^m
 \cong  
 (f_\eta)_*\big(\widetilde{\tau}_*(\widetilde{\tau}^*(\mathcal{O}(m))\big)
 \cong 
 (f_\eta)_*\big(\mathcal{O}(m)\otimes \widetilde{\tau}_*(\mathcal{O}_X)\big)
 \cong 
 (f_\eta)_* \mathcal{O}(m)
 \end{equation}
 on $Y$, where $f_\eta\colon \widetilde{X}_\eta\to Y$ satisfies $f=f_\eta\circ \tau$. Therefore, the section rings of $L_C(\eta)$ on $X$ and $\mathcal{O}(1)$ on $\widetilde{X}_\eta$ are isomorphic, so $\widetilde{X}_\eta\cong \Proj R(X,L_C(\eta))$ as required. 
 
 Part \four\ follows by combining part \two\ with Corollary~\ref{cor:PicN1}.
\end{proof}

\begin{lem}
\label{lem:Ass1all C'}
 If one GIT chamber $C$ satisfies Condition~\ref{cond:GIT}, then every GIT chamber $C'\subset R_C$ satisfies Condition~\ref{cond:GIT}. In particular, the statement of Lemma~\ref{lem:Assumptiongivesnefissemiample} holds for each chamber in $R_C$. 
\end{lem}
\begin{proof}
Conditions~\ref{cond:GIT}(2) and (3) are independent of the choice of chamber $C'\subset R_C$, so it suffices to prove Condition~\ref{cond:GIT}(1) holds for $C'$. There are only finitely many GIT chambers \cite[Theorem~2.4]{Thaddeus96}, so we proceed by induction. We know Condition~\ref{cond:GIT}(1) holds for $C$.  For the induction step, let $C_+$ and $C_-$ be adjacent chambers in $R_C$ separated by a GIT wall, where Condition~\ref{cond:GIT}(1) holds for $C_-$. Let $\theta_+\in C_+$ and $\theta_-\in C_-$. Since the wall separating $C_+$ and $C_-$ lies in the interior of $R_C$, Condition~\ref{cond:GIT}(2) gives that $X_{\theta_+}$ is normal, and Proposition~\ref{prop:floppingwall} implies that the birational map $\psi \colon \Xtheta{_-}\dashrightarrow \Xtheta{_+}$ from diagram \eqref{eqn:VGITdiagram} is an isomorphism in codimension one such that 
   \begin{equation}
  \label{eqn:LC-=LC+}
  L_{C_-}(\eta) \cong\psi^*L_{C_+}(\eta)\quad\text{ for all }\eta\in G^\vee_{\mathbb{Q}}.
  \end{equation}
  Since $X_{\theta_-}$ is $\QQ$-factorial and $L_{C_-}$ is surjective, then $X_{\theta_+}$ is $\QQ$-factorial by Proposition~\ref{prop:floppingwall}. Pushforward along $\psi$ identifies the class groups of $X_{\theta_\pm}$, and hence $\QQ$-factoriality implies that $\psi^*$ identifies the rational Picard groups $\Pic(X/Y)_\QQ\cong \Pic(\Xtheta{_-}/Y)_\QQ\cong \Pic(\Xtheta{_+}/Y)_\QQ$. Since $L_{C_-}$ is an isomorphism, then so is $L_{C_+}$
by \eqref{eqn:LC-=LC+}. 
 This shows that Condition~\ref{cond:GIT}(1) holds for $C_+$, so it holds for each chamber in $R_C$ by induction. This completes the proof of the first statement, while the second statement follows by applying the proof of Lemma~\ref{lem:Assumptiongivesnefissemiample} verbatim to each chamber $C'$ in $R_C$. 
\end{proof}

\begin{remark}
\label{rem:LCcompatible}
The proof of Lemma~\ref{lem:Ass1all C'} shows that for each chamber $C'\subset R_C$ and any $\theta'\in C'$, the linearisation map $L_{C'}$ is equal to $L_C$, up to the identification $\Pic(X/Y)_{\QQ}\cong \Pic(\Xtheta{^\prime}/Y)_{\QQ}$.
\end{remark}

By taking the composition of the linearisation map $L_{C}$ with the isomorphism from Lemma~\ref{lem:Assumptiongivesnefissemiample}\four, we may identify the target of $L_{C}$ with $N^1(X/Y)$ whenever Condition~\ref{cond:GIT} holds.

\begin{thm}
\label{thm:movable}
  Suppose that Condition~\ref{cond:GIT} holds. Then:
  \begin{enumerate}
      \item[\one] the linearisation map $L_{C}\colon G^{\vee}_{\mathbb{Q}} \rightarrow N^1(X/Y)$
 is an isomorphism that identifies the GIT decomposition of the region $R_C$ with the Mori chamber decomposition
 of $\Mov(X/Y)$; and
      \item[\two] for any $\eta\in R_C$, the section ring $R(X,L_C(\eta))$ is a finitely generated $\mathcal{O}_Y$-algebra, and the model $\Proj R(X,L_C(\eta))$ is the variety $\widetilde{X}_\zeta$ appearing in the Stein factorisation of $X_{\theta'} \to X_\zeta$ for $\theta'$ in a GIT chamber $C'$ whose boundary contains $\zeta$.
 \end{enumerate} 
  In particular, every  small birational model of $X$ over $Y$ can be obtained as a GIT quotient of the form $\widetilde{X}_\eta$ for some $\eta$ in the interior of $R_C$. 
\end{thm}

\begin{remark}
Theorem~\ref{thm:movable} establishes that the region $R_C$ from Condition~\ref{cond:GIT} is the convex polyhedral cone $L_{C}^{-1}\big(\Mov(X/Y)\big)$. 

Conversely, for a given GIT set-up, suppose that there is a chamber $C$ with $X=\Xtheta{}$ for $\theta\in C$, such that $L_C$ is an isomorphism of fans between $R_C$ and $\Mov(X/Y)$. 
If, in addition, each VGIT morphism for an interior wall of $R_C$ has unstable locus of codimension at least two, then all parts of Condition~\ref{cond:GIT} are satisfied.
\end{remark}

\begin{proof}
Suppose first that $R_C$ contains a unique GIT chamber, i.e.\ $\overline{C}=R_C$. For every wall of $C$ and any $\theta_0$ that is general in the wall, Condition~\ref{cond:GIT}(3) implies that the Stein factorisation of the induced VGIT morphism $\tau\colon X\to \Xtheta{_0}$ is of fibre type or it is a divisorial contraction, so $\Nef(X/Y)=\Mov(X/Y)$. The identification of $\overline{C}$ with $\Nef(X/Y)$ and the isomorphism from $\widetilde{X}_\eta$ to $\Proj R(X,L_C(\eta))$ for all relevant $\eta\in \overline{C}$ were established in Lemma~\ref{lem:Assumptiongivesnefissemiample}. This proves \one\ and \two\ when $\overline{C}=R_C$.
  
For the general case, we noted in Remark~\ref{rem:LCcompatible} 
 that the linearisation maps $L_{C'}$ for all $C'\subset R_C$ are compatible with taking the strict transform along the appropriate birational map $\psi_{C'}$. Thus, for the chamber $C$ from Condition~\ref{cond:GIT} and for any other chamber $C'\subset R_C$ with $\theta'\in C'$, we obtain $L_C\big(\overline{C^\prime}\big) = \psi_{C'}^*L_{C^\prime}\big(\overline{C^\prime}\big) = \psi_{C'}^*\big(\Nef(\Xtheta{^\prime}/Y)\big)$ by \eqref{eqn:LC-=LC+} and Lemma~\ref{lem:Ass1all C'}. Applying $L_C$ to the obvious decomposition $R_C =  \bigcup_{C^\prime\subset R_C}  \overline{C^\prime}$ gives
 \begin{equation}
     \label{eqn:LCFdecomp}
 L_C(R_C) =  \bigcup_{C^\prime\subset R_C} \psi_{C'}^* \big(\Nef(\Xtheta{^\prime}/Y)\big).
  \end{equation}
  Any wall in the boundary of $L_C(R_C)$ therefore lies in the boundary of $\Nef(\Xtheta{^\prime}/Y)
$ for some $C'\subset R_C$ and $\theta'\in C'$. Condition~\ref{cond:GIT}(3) ensures that if $\theta_0$ is general in the corresponding boundary wall of $R_C$, then the Stein factorisation of the induced VGIT morphism $\tau\colon \Xtheta{^\prime}\to \Xtheta{_0}$  is either of fibre type or it is a divisorial contraction, so $L_C$ sends the boundary of $R_C$ into the boundary of $\Mov(X/Y)$. It follows from \eqref{eqn:LCFdecomp} that $L_C$ maps GIT chambers in $R_C$ to the open ample cones in $\Mov(X/Y)$ of small birational models of
$X$ over $Y$. More generally it maps GIT cones in $R_C$ to cones of the Mori decomposition of $\Mov(X/Y)$. This proves \one.

For \two, let $\eta\in R_C$ and let $C^\prime\subset R_C$ denote any chamber such that $\eta\in \overline{C'}$. Lemma~\ref{lem:Ass1all C'} implies that the codomain of the Stein factorisation of $X_{\theta'}\to X_\eta$ for $\theta'\in C'$ satisfies $\widetilde{X}_\eta\cong X(L_{C'}(\eta))$. Since both $C$ and $C^\prime$ are contained in $R_C$, the linearisation maps $L_{C'}$ and $L_C$ are compatible with taking the strict transform along the appropriate birational map, giving $\widetilde{X}_\eta \cong X(L_{C'}(\eta))\cong X(L_{C}(\eta))$.

For the final statement, $X$ is $\QQ$-factorial and hence so is $X_{\eta}$ for every generic $\eta\in R_C$ by Lemma~\ref{lem:Ass1all C'}. Let $X'$ be a small birational model of $X$ over $Y$. Then there is a line bundle $L'$ in the interior of $\Mov(X/Y)$ satisfying $X(L')\cong X'$. The character $\theta':=L_C^{-1}(L')$ lies in the interior of $R_C$ and satisfies $\widetilde{X}_{\theta'}\cong X'$ as required.
\end{proof}

\begin{example}
In Theorem~\ref{thm:movable}, if $X$ were not $\QQ$-factorial, then it does not necessarily follow that every $\QQ$-factorial small 
birational model of $X$ can be obtained by variation of GIT quotient. For example, if $G$ is trivial and $V$ is the locus of square-zero $n \times n$ matrices, then $X=Y=V$. Condition~\ref{cond:GIT} applies because $\Pic(X)$ is trivial and hence $G^\vee_{\QQ}\cong N^1(X/Y)=\{0\}$, so Theorem~\ref{thm:movable} applies. However, $X$ admits a projective crepant resolution $T^* \mathbb{P}^{n-1}$ which is a $\QQ$-factorial small birational model that is not obtained by varying the trivial quotient. In this case, this resolution is a quiver variety and  can be obtained by varying a different GIT quotient; see Corollary~\ref{cor:resgivenbyquiver}.
\end{example}

 This result allows us to draw conclusions about the GIT quotients directly from known results in birational geometry as follows.

\begin{cor}
\label{cor:sbm}
Suppose that Condition~\ref{cond:GIT} holds, and let $\theta, \theta'\in R_C$. Then:
\begin{enumerate}
    \item[\one] $\Xtheta{}$ is isomorphic to $\Xtheta{'}$ over $Y$ if and only if $\theta, \theta'$ lie in the same GIT cone; 
   \item[\two] if $\theta$ is generic and $\theta'$ in the interior of $R_C$, then the normalisation $\widetilde{X}_{\theta'}$ 
   of $\Xtheta{'}$ is the (not necessarily $\mathbb{Q}$-factorial) small birational model of $\Xtheta{}$ over $Y$ given by the line bundle $L_C(\theta')$;
   \item[\three] for any $\theta$ in the interior of $R_C$, the variety $\Xtheta{}$ is $\QQ$-factorial if and only if $\theta$ is generic;
   \item[\four] the dimension of $N^1(\widetilde{X}_{\eta}/Y)$ equals the dimension of the minimal GIT cone containing $\eta$, where for any GIT chamber $C'$ with $\eta\in \overline{C'}$ and $\theta'\in C'$, we write $X_{\theta'}\to \widetilde{X}_{\eta}$ for the Stein factorisation of the VGIT morphism $X_{\theta'}\to X_{\eta}$.
     \end{enumerate}
\end{cor}
\begin{proof}
 For \one, the models $X(L)$ and $X(L')$ associated to $L, L'\in \Pic(X/Y)$ are isomorphic over $Y$ if and only if $L, L'$ lie in the same face of the decomposition of $\Mov(X/Y)$ from \eqref{eqn:LCFdecomp}. Thus, \one\ follows from Theorem~\ref{thm:movable}.
 
 For \two, the interior of $R_C$ is identified with the interior of $\Mov(X/Y)$, so the GIT quotients $\Xtheta{'}$ and $\Xtheta{}$ are isomorphic in codimension one over $Y$.  Moreover, if $C' \subseteq R_C$ is a GIT chamber with $\theta' \in \overline{C'}$, then the codomain of the Stein factorisation $\Xtheta{''} \to \widetilde{X}_{\theta'}$ of the 
 VGIT morphism for $\theta'' \in C'$ is just the normalisation of $\Xtheta{'}$, as noted after \eqref{eqn:flopStein}. The result then follows from Theorem~\ref{thm:movable}\two.

 For \three, one direction was noted in Lemma~\ref{lem:Ass1all C'} while for the other, it is well-known that the base of a flip is not $\QQ$-factorial. Indeed, if $\Xtheta{_0}$ were $\QQ$-factorial, then in the notation of Proposition~\ref{prop:floppingwall}, any Cartier divisor $D$ satisfying $\mathcal{O}_{\Xtheta{_-}}(D)=L_{C_-}(\theta_+)$ would define a Weil divisor $(\tau_-)_*(D)$ on $\Xtheta{_0}$, making $mD$ and $(\tau_-)^*(\tau_-)_*(mD)$ linearly equivalent for some $m>0$. However, the intersection numbers of these divisors with respect to a curve contracted by $\tau_-$ are negative and zero respectively.
 
 For \four,  let $F$ be the minimal GIT cone containing $\eta$, let $C'\subset R_C$ be a chamber containing $F$ in its closure, and let $\theta'\in C'$. Lemma~\ref{lem:Ass1all C'} implies that $L_{C}$ identifies  $F$ with the minimal face of $\Nef(\Xtheta{^\prime}/Y)$ containing $L_C(\theta)$. By \eqref{eqn:ellC}, every line bundle in the interior of this face is the pullback of an ample bundle via the Stein factorisation $\widetilde{\tau}\colon X_{\theta'}\to \widetilde{X}_\eta$ of the VGIT morphism $\tau\colon \Xtheta{^\prime}\to X_\eta$, so $\dim L_C(F) \leq \dim (\widetilde{\tau})^*N^1(\widetilde{X}_{\eta}/Y)$. On the other hand, $L_C(F)$ is dual via the intersection pairing to the face $\sigma$ of the Mori cone of curves generated by the numerical classes of curves contracted by $\widetilde{\tau}$, so $\dim L_C(F) = \dim (\sigma^\perp)$ for $\sigma^\perp =\{L\in N^1(\widetilde{X}_{\theta'}/Y) \mid \deg L\vert_\ell =0 \;\forall \; \ell\in \sigma\}$. The pullback via $\widetilde{\tau}$ of any class in $N^1(\widetilde{X}_\eta/Y)$ has degree zero on each generator of $\sigma$, so $(\widetilde{\tau})^*N^1(\widetilde{X}_\eta/Y)\subseteq \sigma^\perp$ and hence $\dim (\widetilde{\tau})^*N^1(\widetilde{X}_\eta/Y)\leq \dim (\sigma^\perp) = \dim L_C(F)$. The map $\tau^*$ is injective, 
 so $\dim F = \dim L_C(F) = \dim N^1(\widetilde{X}_\eta/Y)$ as required.
 \end{proof}

 \subsection{Relative Mori Dream Spaces}
Example~\ref{exa:MDS} illustrates that the GIT construction of any Mori Dream Space via its Cox ring gives rise to a GIT chamber that satisfies our Condition~\ref{cond:GIT}. The next result provides a partial converse, but we emphasise that even for a Mori Dream Space, we are typically interested in applying our Condition~\ref{cond:GIT} for new GIT descriptions that do not involve the Cox ring directly.

\begin{cor}
\label{cor:mds}
 Suppose that Condition~\ref{cond:GIT} holds. Then the GIT quotient $X=\Xtheta{}$ is a Mori Dream Space over $Y$, i.e.:
 \begin{enumerate}
 \item[\one] $X$ is normal and $\QQ$-factorial;
 \item[\two] $\Pic(X/Y)_\QQ\cong N^1(X/Y);$
 \item[\three] the relative nef cone $\Nef(X/Y)$ is generated by finitely many semiample
line bundles; and
 \item[\four] there exists $k\geq 0$ and  $\QQ$-factorial varieties $X=X_0, X_1, \dots, X_k$, each projective over $Y$, as well as birational maps $\psi_i\colon X\dashrightarrow X_i$ over $Y$ for $0\leq i\leq k$ that are isomorphisms in codimension one, such that  
 \begin{equation}
 \label{eqn:movconedecomp}
\Mov(X/Y) = \bigcup_{0\leq i\leq k} \psi_i^*\;\Nef(X_i/Y),
\end{equation}
 where each cone in this description is generated by finitely many semiample line bundles.
 \end{enumerate}
\end{cor}
 \begin{proof}
 Let $C^\prime\subset R_C$ be the chamber with $\theta\in C^\prime$. Part \one\ holds by assumption, while part \two\ is Lemma~\ref{lem:Assumptiongivesnefissemiample}\four. For part \three, the closure $\overline{C'}$ is a polyhedral cone \cite[Theorems~2.3-2.4]{Thaddeus96}, and hence so is $\Nef(X/Y)$ by Lemma~\ref{lem:Ass1all C'}. 
 Any choice of cone generators for $\Nef(X/Y)$ are semiample over $Y$ by Lemma~\ref{lem:Ass1all C'}. 
 Part \four\ follows from the decomposition \eqref{eqn:LCFdecomp} and the equality of cones $L_C(R_C)=\Mov(X/Y)$ from Theorem~\ref{thm:movable}. 
 \end{proof}

 \begin{remark}
 \begin{enumerate}
     \item The birational maps that feature in Corollary~\ref{cor:mds} are all constructed by variation of GIT quotient, so we need not appeal to the existence of flips from~\cite{BCHM}.
 \item It is instructive to compare Condition~\ref{cond:GIT} with the criteria for a GIT quotient to be a Mori Dream Space given by Hu and Keel~\cite[Lemma~2.2]{HuKeel00} (see also Ohta~\cite[Theorem~6.7]{Ohta20}). While our Condition~\ref{cond:GIT}(1) is equivalent to their third and fourth criteria, our criteria (2) and (3) differ considerably from their first and second criteria. We show in Theorem~\ref{thm:mainquiverVars} that quiver varieties satisfy our Condition~\ref{cond:GIT}, so they are Mori Dream Spaces by Corollary~\ref{cor:mds}. However, the next example shows that even the simplest quiver varieties can fail to satisfy the Hu and Keel criteria.
 \end{enumerate}
 \end{remark}
 
\begin{example}
\label{exa:A1}
 The minimal resolution of the $A_1$ singularity $Y=\mathbb{V}(uv-w^2)\subset \mathbb{A}^3$ is obtained by variation of GIT quotient for quiver varieties associated to the graph with one node, and vectors $\bv=1$, $\bw = 2$. The quiver $Q$ is the McKay quiver for the cyclic group of order two in $\SL(2,\kk)$, and  dimension vector $\alpha=(1,1)$. In this case, $V = \mathbb{V}(ad-bc)\subset \mathbb{A}^4$ admits an action by $G=(\kk^\times)^2/\kk^\times$, and for $\theta = (-1,1)\in G^\vee$, the $\theta$-unstable locus is the intersection of $V$ with $\mathbb{V}(a,b)$. This locus is of codimension one in $V$, so $X=V\git_\theta \, G$ fails to satisfy the Hu and Keel criterion \cite[Lemma~2.2(1)]{HuKeel00}. However, this example satisfies Condition~\ref{cond:GIT} (see Theorem~\ref{thm:mainquiverVars}), so it is a Mori Dream Space over $Y$ by Corollary~\ref{cor:mds}.  
\end{example}

\begin{remark}
 If Condition~\ref{cond:GIT} holds, then combining Corollary~\ref{cor:mds} with the statement of Hu and Keel~\cite[Proposition~2.9]{HuKeel00} (see also Ohta~\cite[Proposition~6.9]{Ohta20}) implies that the Cox ring of the variety $\Xtheta{}$ is a finitely generated $\kk$-algebra for $\theta\in C$. Our Theorem~\ref{thm:movable} allows for an alternative description of the Cox ring of $X_{\theta}$, see \cite{BCSCox}.
\end{remark}

\subsection{Strong convexity and injectivity}
In practice, injectivity of $L_C$ can often follow from the following simple criterion.

A cone is called \emph{strongly convex} if the origin is a face. Given a fan, the origin is a cone if and only if all cones are strongly convex. 

Given a chamber $C$, let $K_C$ be the intersection of all supporting hyperplanes of the walls of $C$.  The following lemma is standard.
\begin{lem}
The chamber $C$ is a product of $K_C$ and a strongly convex cone.  In particular, $K_C$ is the largest linear subspace of $C$, the intersection of all faces of $C$, and the only face which is a linear subspace.
 \end{lem}
\begin{proof}
Since all faces of $C$ are defined by intersections of half-spaces supported on the hyperplanes defining $K_C$, we see that $K_C$ is contained in all faces. This also shows that it is the largest linear subspace of $C$ and the intersection of all faces. For the first statement, $C/K_C$ is a strongly convex cone, and by taking a splitting of the ambient vector space, $C \cong C/K_C \times K_C$. 
\end{proof}
In particular, $K_C$ does not depend on the choice of chamber, and it is zero if and only if all cones are strongly convex. Let us simply call it $K$ from now on.

\begin{lem}
\label{lem:fixedpoint}
Suppose that $V$ has a $G$-fixed point. Then each GIT cone is strongly convex.
\end{lem}
\begin{proof} 
By definition, if $\theta$ is a nontrivial effective character, then any $G$-fixed point is $\theta$-unstable. So $V^{\theta}\neq V=V^0$, and hence $\{0\}$ is a GIT cone. 
\end{proof}
In particular, if $V$ is conical and the $G$ action is equivariant, then the origin is a $G$-fixed point.

\begin{prop}
\label{prop:stronglyconvex}
Suppose that $C$ is a GIT chamber such that the VGIT morphism $\tau$ from \eqref{eqn:VGITdiagram} obtained by choosing general $\theta_0$ in any wall of $\overline{C}$ contracts at least one curve. Then $\ker(L_{C}) \subseteq K$.
\end{prop}
\begin{proof}
We adapt the proof of \cite[Proposition 6.1]{BellamyCraw}.  Suppose on the contrary that there is a vector $\eta \in \ker(L_{C})$ not contained in $K$. Then $\eta$ is not contained in every supporting hyperplane of $\overline{C}$. We may translate any $\theta \in C$ by a rational multiple of $\eta$ to hit a wall of $\overline{C}$, say at $\theta_0$.  Since $\eta \in \ker(L_{C})$, the line bundle $L_C(\theta) = L_C(\theta_0)$ is ample, so some multiple is very ample. However, this is a contradiction because the VGIT morphism $\tau\colon \Xtheta{}\to\Xtheta{_0}$ contracts a curve.
\end{proof}

\begin{cor}\label{c:git-stronglyconvex}
Suppose that $\overline{C}$ is strongly convex. If, for every wall of $\overline{C}$, the VGIT morphism $\tau$ from \eqref{eqn:VGITdiagram} obtained by choosing $\theta_0$ to be general in the wall contracts at least one curve, then $L_C$ is injective.
\end{cor}

\begin{rem}\label{r:image-unstable}
In the situation of diagram \eqref{eqn:flop}, suppose $\tau_\pm$ are both small contractions. If the image under both $\tau_\pm$ of the unstable locus $\Uns(\tau_\pm)$ is singular, and the singular locus of $X_{\theta_0}$ has codimension at least two, then each of $\Uns(\tau_\pm)$ has codimension at least two in $\Xtheta{_\pm}$. This can be a convenient condition to check in practice. 
\end{rem} 

Putting these together, we conclude the following useful criteria for Condition~\ref{cond:GIT} to hold:

\begin{cor}\label{c:useful-criterion}
 Suppose that there exists a GIT chamber $C$ such that:
\begin{enumerate}
    \item[\emph{(1)}] the closed cone $\overline{C}$ is strongly convex, the map $L_C$ is surjective, and for $\theta\in C$, the GIT quotient $X:=\Xtheta{}$ is a 
   $\QQ$-factorial,
    normal variety;
    \item[\emph{(2)}] for any GIT chambers $C_\pm$ in $R_C$ sharing a wall, the VGIT morphisms $\tau_{\pm}$ from \eqref{eqn:flop} both contract a curve and they each map their unstable locus to the singular locus of $\Xtheta{_0}$ which is of codimension at least two; and
    \item[\emph{(3)}]  each boundary wall of $R_C$ is either divisorial or of fibre type.
\end{enumerate}
 Then Condition~\ref{cond:GIT} is satisfied.
\end{cor}

\subsection{Minimal models of Gorenstein singularities}
In this section we assume, in addition, that $Y$ has Gorenstein singularities and that $f\colon X \to Y$ is a projective crepant resolution, or more generally, a projective $\QQ$-factorial terminalisation, that is, a projective crepant birational morphism from a variety with $\QQ$-factorial terminal singularities.

Suppose that there exists another projective $\QQ$-factorial terminalisation
$f'\colon X'\to Y$. Then $X$ and $X'$ are birational minimal models over $Y$, so by \cite[Theorem~3.52]{KollarMori} there is a movable line bundle $L$ on $X$ such that $X'\cong X(L)$, and the morphism $f'=f_{L}$ fits into a commutative diagram \eqref{eqn:psiDXD}. In particular, if there is a GIT construction such that $X\cong\Xtheta{}$ for $\theta\in C$ as in Condition~\ref{cond:GIT}, then Theorem~\ref{thm:movable} implies that there is a chamber $C^\prime$ in the GIT region $R_C$ such that $X^\prime \cong \Xtheta{^\prime}$ for $\theta^\prime\in C^\prime$. More generally, we have the following.

\begin{cor}\label{c:partial-cr}
 Suppose that Condition~\ref{cond:GIT} holds and that $f_\theta\colon \Xtheta{}\to Y$ is a projective $\QQ$-factorial terminalisation
 for $\theta\in C$. 
 \begin{enumerate}
     \item[\one] If a projective, crepant morphism $g \colon Z \to Y$ is dominated by a projective $\QQ$-factorial terminalisation
     $f'\colon X'\to Y$, then there is a chamber $C^\prime$ in $R_C$ and $\eta\in \overline{C'}$ such that $Z\cong X_{\eta}$ and $g=f_\eta$. Moreover, $Z$ has terminal singularities if and only if $\eta$ lies in the interior of $R_C$. 
     \item[\two] Conversely, for all $\eta\in R_C$, the morphism $f_\eta\colon X_\eta\to Y$ is projective  and crepant. 
 \end{enumerate}
 \end{cor}
\begin{proof}
 For \one, the choice of $C^\prime\subseteq R_C$ is described in the paragraph preceding Corollary~\ref{c:partial-cr}. Since $f^\prime$ factors via $g$, the morphism $h\colon X'\to Z$ satisfying $f^\prime = g\circ h$ is obtained from a basepoint-free line bundle $L\in \Nef(X'/Y)$. The first statement follows from Theorem~\ref{thm:movable} by setting $\eta:=L_{C}^{-1}(L)$. For the second statement,  $Z$ fails to be terminal if and only if the crepant morphism $h\colon X\to Z\cong X_\eta$ contracts a divisor. This holds if and only if $\eta$ lies in the boundary of $\Mov(X'/Y)$, which is identified with $\Mov(X/Y)$ by pullback along the birational map $\psi$ from \eqref{eqn:psiXprime}.
 
 For \two, each $\eta\in R_C$ lies in the closure of some chamber $C^\prime\subseteq R_C$, so for $\theta^\prime\in C^\prime$, the morphism $f_{\theta^\prime}\colon \Xtheta{^\prime}\to Y$ factors via $f_\eta\colon X_\eta\to Y$. Since $f_{\theta^\prime}$ is crepant, then so too is $f_\eta$.
\end{proof}

\begin{remark}
 The hypothesis from Corollary~\ref{c:partial-cr}\one, namely that a projective, crepant morphism $g\colon Z\to Y$ is dominated by a projective, $\mathbb{Q}$-factorial terminalisation $f^\prime\colon X^\prime\to Y$, is superfluous in light of results from \cite{BCHM}. However, we choose to leave this as an explicit assumption here to underline the fact that the results in the current paper do not rely on \cite{BCHM} in any way.
\end{remark}

\subsection{Weakening the hypotheses}

 We conclude this section by noting that our results hold in a slightly more general context. 
 
 Rather than assume that $L_C$ is an isomorphism in Condition~\ref{cond:GIT}(1), suppose instead that there exists an affine subspace $\Lambda \subseteq G^\vee_\QQ$ that is not contained in the linear span of any GIT wall, such that the restriction of  $L_C$ to $\Lambda$ is an affine isomorphism $\Lambda\cong \Pic(X/Y)_{\QQ}$. If one replaces the study of chambers $C$ in $G^\vee_\QQ$ by the open cones $C\cap \Lambda$ in $\Lambda$, and the closed GIT region $R_C$ by $R_C\cap \Lambda$, then the proof of Theorem~\ref{thm:movable} applies verbatim if one assumes the analogue of Condition~\ref{cond:GIT} in this context.
 
 \begin{example}
 \label{exa:framedMcKayslice}
 If $\Xtheta{}$ denotes the quiver variety for a framed extended Dynkin quiver of dimension vector $(1,\delta)$ where $\delta$ denotes the minimal imaginary root for an affine root system of type ADE as in  \cite[Proposition~7.11\one]{BellamyCraw}, then any affine hyperplane $\Lambda$ orthogonal to the kernel of $L_C$ will do.
 \end{example}

 One might hope to weaken Condition~\ref{cond:GIT}(1) further to require only surjectivity of $L_C$. However, as Corollary~\ref{c:git-stronglyconvex} shows, if $L_C$ is not injective, then there are two cases to consider. Firstly, if $\overline{C}$ is not strongly convex, then it is natural to study a transverse slice $\Lambda$ to the maximal linear subspace in $\overline{C}$, as described above. The second case is where $C$ has a wall for which the VGIT morphism does not contract a curve. In this latter case, one might hope to restrict attention to an affine subspace that intersects none of these walls (such walls are called `fake' or `type $\0$' in the literature); this is precisely the situation described in Example~\ref{exa:framedMcKayslice} above.  

\section{Nakajima quiver varieties}
In this section we establish our main result for Nakajima quiver varieties, see Theorem~\ref{thm:mainquiverVars}.

\subsection{Quiver varieties}

Choose an arbitrary finite graph with vertices $0,\ds,r$ and let $H$ be the set of pairs consisting of an edge, together with an orientation on it. Let $\tail(a)$ and $\head(a)$ denote the tail and head respectively of the oriented edge $a \in H$. Let $a^*$ denote the same edge, but with opposite orientation. We fix an orientation of the graph, that is, a subset $\Omega \subset H$ such that $\Omega \cup \Omega^* = H$ and $\Omega \cap \Omega^* = \emptyset$. Then $\epsilon : H \rightarrow \{ \pm 1 \}$ is defined to take value $1$ on $\Omega$ and $-1$ on $\Omega^*$. 

Fix collections $V_0, \ds, V_r$ and $W_0, \ds, W_r$ of finite-dimensional vector spaces over $\kk$ and set 
$$
\bv = (\dim V_0, \ds, \dim V_r), \quad \bw = (\dim W_0, \ds, \dim W_r).
$$
The group $G(\bv) := \prod_{k = 0}^r \GL(V_k)$ acts naturally on the space
$$
\mathbf{M}(\bv,\bw) := \left(\bigoplus_{a \in H} \Hom_{\kk}(V_{\tail(a)}, V_{\head(a)})\right) \oplus \left(\bigoplus_{k = 0}^r \big(\Hom_{\kk}(V_k,W_k) \oplus \Hom_{\kk}(W_k,V_k)\big)\right).
$$
This action of $G(\bv)$ is Hamiltonian for the natural symplectic structure on $\mathbf{M}(\bv,\bw)$ and, after identifying the dual of $\mf{g}(\bv) := \mathrm{Lie} \ G(\bv)$ with $\mf{g}(\bv)$ via the trace pairing, the corresponding moment map $\gbf{\mu} \colon \mathbf{M}(\bv,\bw) \rightarrow \mf{g}(\bv)$ satisfies
$$
\gbf{\mu}(B,i,j) = \left( \sum_{\head(a) = k} \epsilon(a) B_a B_{a^*} + i_k j_k \right)_{k = 0}^r
$$
where $i_k \in \Hom_{\kk}(W_k,V_k), j_k \in \Hom_{\kk}(V_k,W_k)$ and $B_a \in \Hom_{\kk}(V_{\tail(a)},V_{\head(a)})$. Though one can talk about arbitrary stability conditions in this context, as was done in \cite{NakajimaBranching}, it is easier in our case to apply the trick of Crawley-Boevey~\cite{CBmomap} and reduce to the case where each $W_k = 0$ by introducing a framing vertex. 

The set $H$ associated to the graph can be thought of as the arrow set of a quiver. We frame this quiver by adding an additional vertex $\infty$, as well as $\bw_i$ arrows from vertex $\infty$ to vertex $i$ and another $\bw_i$ arrows from vertex $i$ to vertex $\infty$. This framed (doubled) quiver is denoted $Q = (Q_0,Q_1)$, where $Q_0 = \{ \infty, 0, \ds, r \}$. Each dimension vector $\bv = (\dim V_0, \ds, \dim V_r)$ for the original graph determines a dimension vector $\alpha:= (1,\bv) = (1,\dim V_0, \ds, \dim V_r)$ for $Q$. We may identify $\mathbf{M}(\bv,\bw)$ with the space
\[
\Rep(Q,\alpha):= \bigoplus_{a\in Q_1} \Hom_{\kk}(\kk^{\alpha_{\tail(a)}},\kk^{\alpha_{\head(a)}})
\]
of representations of $Q$ of dimension vector $\alpha$ in such a way that the $G(\bv)$-action on $\mathbf{M}(\bv,\bw)$ corresponds to the action of the group $G(\alpha):=\big(\prod_{i\in Q_0} \GL(\alpha_i)\big)/\kk^\times$ on $\Rep(Q,\alpha)$ by conjugation and, moreover, the above map $\gbf{\mu}$ corresponds to the moment map $\mu$ induced by this $G(\alpha)$-action on $\Rep(Q,\alpha)$. 

 From now on, we assume that $\mathbf{v}$ satisfies $v_i\neq 0$ for all $0\leq i\leq r$, and $\mathbf{w} \neq 0$. This is equivalent to choosing a dimension vector $\alpha$ for $Q$ with component $\alpha_{\infty} = 1$ and $\alpha_i \neq 0$ for all $i$ (and $Q$ must have at least one arrow from some vertex $i$ to vertex $\infty$). Then the rational vector space
\[
\Theta := \big\{\theta \in \Hom(\Z^{Q_0},\Q) \mid  \theta(\alpha)=0\big\}
\]
 satisfies $G(\alpha)^{\vee} \o_{\Z} \Q = \Theta$, where $\chi_\theta(g) = \prod_{i \in Q_0} \det(g_i)^{\theta_i}$ for $g\in G(\alpha)$. For $\theta\in \Theta$, after replacing $\theta$ by a positive multiple if necessary, the (Nakajima) \emph{quiver variety} associated to $\theta$ is the categorical quotient 
 \[
 \mathfrak{M}_\theta(\bv,\bw) :=  \mu^{-1}(0)\git_{\chi_\theta} G(\alpha)= \mu^{-1}(0)^{\theta}/\!\!/G(\alpha) = \Proj\bigoplus_{k\geq 0} \kk[\mu^{-1}(0)]^{\chi_{k\theta}},
 \]
where $\mu^{-1}(0)^{\theta}$ denotes the locus of $\chi_\theta$-semistable points in $\mu^{-1}(0)$ and $\kk[\mu^{-1}(0)]^{\chi_{k\theta}}$ is the $\chi_{k\theta}$-semi-invariant slice of the coordinate ring of the locus $\mu^{-1}(0)$. Note that $\kk^\times$ acts on $\mathbf{M}(\bv,\bw)$ by scaling, and this action descends to an action on $\mathfrak{M}_\theta(\bv,\bw)$.
  
 For a more algebraic description of $\mathfrak{M}_\theta(\bv,\bw)$,  extend $\epsilon$ to $Q$ by setting $\epsilon(a) = 1$ if $a \colon \infty \rightarrow i$ and $\epsilon(a) = -1$ if $a \colon i \rightarrow \infty$. The \textit{preprojective algebra} $\Pi$ is the quotient of the path algebra $\kk Q$ by the relation
 \begin{equation}
 \label{eqn:preprojective}
\sum_{a \in Q_1} \epsilon(a) aa^* = 0. 
\end{equation}
Given $\theta\in \Theta$, we say that a $\Pi$-module $M$ of dimension vector $\alpha$ is $\theta$-semistable if $\theta(N)\geq 0$ for all submodules $N\subseteq M$, and it is $\theta$-stable if $\theta(N)>0$ for all proper nonzero submodules. Two $\theta$-semistable $\Pi$-modules are $\Seshadri$-equivalent if their composition series agree in the abelian category of $\theta$-semistable $\Pi$-modules. A finite dimensional $\Pi$-module is said to be $\theta$-polystable if it is a direct sum of $\theta$-stable $\Pi$-modules. King~\cite{KingStable} proved that a $\Pi$-module $M$ of dimension vector $\alpha$ is $\theta$-semistable (resp.\ $\theta$-stable) if and only if the corresponding point of $\mu^{-1}(0)$ is $\chi_\theta$-semistable (resp.\ $\chi_\theta$-stable) in the sense of GIT. In fact \cite[Propositions~3.2, 5.2]{KingStable} establishes that the quiver variety $\mathfrak{M}_\theta(\bv,\bw)$ is the coarse moduli space of $\Seshadri$-equivalence classes of $\theta$-semistable $\Pi$-modules of dimension vector $\alpha$, where the closed points of $\mathfrak{M}_\theta(\bv,\bw)$ are in bijection with the $\theta$-polystable representations of $\Pi$ of dimension $\alpha$. We write $\mathfrak{M}_\theta(\bv,\bw)^s$ for the (possibly empty) open subset of $\mathfrak{M}_\theta(\bv,\bw)$ parametrising $\theta$-stable representations.

\subsection{Kirwan surjectivity}
 Recall that $\bw\neq 0$ and $v_i \neq 0$ for all $i$, with  $\alpha:= (1,\bv)\in \mathbb{N}^{Q_0}$. In this case, there exist $\theta$ that are non-degenerate in the sense of \cite[Definition~3.1]{QuiverKirwan}, but $\mu^{-1}(0)^{\theta}$ can be empty.  Let us first recall the condition on $\alpha$ which excludes this.

Associated to the quiver $Q$ is a root system $R$ with positive roots $R^+ = R \cap \ZZ_{\ge 0}^{Q_0}$.  We set $R^+_{\theta} = \{ \gamma \in R^+ \, | \, \theta(\gamma) = 0 \}$. 

\begin{lem}\label{lem:genericfreeaction}
The following are equivalent:
\begin{enumerate}
    \item[\one] $\alpha \in R^+$;
    \item[\two] each $\theta\in \Theta$ is effective, i.e.\ $\mu^{-1}(0)^{\theta} \neq \emptyset$;
    \item[\three] $G(\alpha)$ acts freely on $\mu^{-1}(0)^{\theta}$ for $\theta$ in some dense open subset of $\Theta$. 
\end{enumerate}
Moreover, when these conditions hold, the GIT chambers are precisely the interiors of the top dimensional cones of the GIT fan.
\end{lem}

\begin{proof}
The space $\mu^{-1}(0)^{\theta}$ is non-empty if and only if the quiver variety $\mf{M}_{\theta}(\bv,\bw)$ is non-empty if and only if $\alpha \in \N R^+_{\theta}$. Since $\alpha_{\infty} = 1$, we have $\alpha \in \N R^+_{\theta}$ for all $\theta \in \Theta$ if and only if $\alpha \in R^+$. This shows $\one \iff \two$. For a general $\theta \in \Theta$, the condition $\alpha$ indivisible implies that every $\theta$-semistable representation is $\theta$-stable; in other words, a general $\theta$ is generic. Therefore, $G(\alpha)$ will act freely if $\mu^{-1}(0)^{\theta} \neq \emptyset$, so \two\ implies \three. Conversely, if $G(\alpha)$ acts freely on $\mu^{-1}(0)^{\theta}$ then the latter must be non-empty by definition. But $\mu^{-1}(0)^{\theta} \subset \mu^{-1}(0)^{\theta_0}$ for $\theta \in C$ and $\theta_0 \in \overline{C}$, so $\three$ implies $\two$. 

In the preceding paragraph, we noted that each general $\theta \in \Theta$ is generic. This is precisely the final statement.
\end{proof}
By \cite[Theorem 1.15]{BellSchedQuiver}, the smooth locus of $\mf{M}_{\theta}(\bv,\bw)$ equals the \emph{canonically polystable} locus, which is the locus where the decomposition of polystable representations into a sum of stable ones is of generic type. In particular, if the stable locus is nonempty, then the smooth locus equals the stable locus $\mathfrak{M}_\theta(\bv,\bw)^s$.  

For $\theta \ge \theta_0$, the morphism $f_{\theta} \colon \mf{M}_{\theta}(\bv,\bw) \to \mf{M}_{\theta_0}(\bv,\bw)$ need not be surjective. However, it is birational onto its image \cite[Theorem~A.1]{BellamyCraw}. Hence it is not of fibre type. If the stable locus with respect to $\theta_0$ is empty then it can happen that $f_{\theta}$ is an isomorphism onto (the normalisation of) its image, showing that `fake' walls (also known as walls of `type $\0$', see section~\ref{sec:threefoldsGITwalls}) exist. 

As a consequence of Kirwan surjectivity, established in \cite{QuiverKirwan}, we note that:

\begin{thm}[McGerty--Nevins]\label{thm:Kirwansurjectivequiver}
Assume that $\bw \neq 0$ and the equivalent conditions of Lemma~\ref{lem:genericfreeaction} hold. For any chamber $C$ and for $\theta\in C$, the restriction of the linearisation map defines a surjective map of lattices $\ell_C \colon G^{\vee} \to \Pic\big(\mf{M}_{\theta}(\mathbf{v}, \mathbf{w})/\mf{M}_{0}(\mathbf{v}, \mathbf{w})\big) \cong \Pic\big(\mf{M}_\theta(\mathbf{v},\mathbf{w})\big)$.
\end{thm}

\begin{proof}
Note that $\mf{M}_{0}(\mathbf{v}, \mathbf{w})$ is an affine cone, so $\Pic(\mf{M}_{0}(\mathbf{v}, \mathbf{w})) = 0$ 
and hence the relative Picard group $\Pic(\mf{M}_{\theta}(\mathbf{v}, \mathbf{w})/\mf{M}_{0}(\mathbf{v}, \mathbf{w}))$ equals $\Pic(\mf{M}_{\theta}(\mathbf{v}, \mathbf{w}))$. There is a commutative diagram
\[
\begin{tikzcd}
G^{\vee} \ar[rr,"\ell_C"] \ar[dr] & & \Pic(\mf{M}_{\theta}\big(\mathbf{v}, \mathbf{w})\big) \ar[dl,"c_1"] \\
& H^2\big(\mf{M}_{\theta}\big(\mathbf{v}, \mathbf{w}),\Z\big). &
\end{tikzcd}
\]
The main result of \cite{QuiverKirwan} says that the map $G^{\vee} \to H^2(\mf{M}_{\theta}(\mathbf{v}, \mathbf{w}),\Z)$ is surjective. Since $\mf{M}_{\theta}(\mathbf{v}, \mathbf{w})$ is smooth, \cite[Theorem~7.3.5]{NakJAMS} says that the cycle map $\Pic(\mf{M}_{\theta}(\bv,\bw)) 
 \to H^2(\mf{M}_{\theta}(\mathbf{v}, \mathbf{w}),\Z)$ is an isomorphism. We note that the paper \cite{NakJAMS} assumed that the quiver $Q$ has no loops. However, the proof of \cite[Theorem~7.3.5]{NakJAMS} does not require this restriction. 
\end{proof}
The map $L_C$ is obtained from $\ell_C$ by tensoring by $\mathbb{Q}$, so is also surjective under the hypotheses of the theorem.
\begin{rem}
\label{rem:nontrivialkerLC}
The map $\ell_C$ will have a non-trivial kernel in general. A simple example is given by taking a framed affine Dynkin quiver and $\alpha = e_{\infty} + \delta$; see \cite[Proposition~7.11]{BellamyCraw}. In fact, one can choose suitable $(\Qu,\mathbf{v}, \mathbf{w})$ and generic $\theta$ such that the quiver variety $\mf{M}_{\theta}(\mathbf{v}, \mathbf{w})$ is a minimal resolution of the corresponding Kleinian singularity and $\ell_C$ has kernel of arbitrarily large dimension. 
\end{rem}

\subsection{Applying the main result}
\label{sec:quiverfan}
Let $\langle - , -\rangle$ denote the Ringel form on $\Z^{Q_0}$ and $( - , - )$ its symmetrisation. For any $\gamma\in R$, define $p(\gamma) := 1 - \frac{1}{2}(\gamma,\gamma)$.
\begin{defn}\label{def:Sigma0}
Define $\Sigma_{\theta}$ to be the set of $\gamma \in R^+_{\theta}$ such that 
\[
p(\gamma) > p\big(\beta^{(1)}\big) + \cdots + p\big(\beta^{(k)}\big) 
\]
for every proper decomposition $\gamma = \beta^{(1)} + \cdots + \beta^{(k)}$ with $\beta^{(i)} \in R^+_{\theta}$.
\end{defn}

Crawley-Boevey~\cite[Theorem~1.2]{CBmomap} showed that $\alpha \in \Sigma_{0}$ if and only if there exists a simple (= $0$-stable) $\Pi$-module of dimension vector $\alpha$. More generally, it is shown in \cite[Theorem~1.3]{BellSchedQuiver} that there exists a $\theta$-stable $\Pi$-module of dimension $\alpha$ if and only if $\alpha \in \Sigma_{\theta}$.  
Since $\alpha=(1,\bv)$ is indivisible, every $\theta$-semistable representation will be $\theta$-stable if $\theta(\beta) \neq 0$ for all roots $\beta < \alpha$. That is, the GIT walls are contained in the union of hyperplanes $\beta^\perp$. For a precise description of the GIT walls, see Section \ref{s:combinatorics} below.

\begin{lem}
\label{lem:contractscurve}
 Assume $\alpha\in \Sigma_0$. For any chamber $C$, let $\theta\in C$ and choose $\theta_0\in \overline{C}\setminus C$. Then the surjective, birational VGIT morphism $\tau \colon \mf{M}_{\theta}(\bv,\bw) \to \mf{M}_{\theta_0}(\bv,\bw)$ contracts at least one curve.
\end{lem}
\begin{proof}
Since $\alpha \in \Sigma_0$, the simple locus ($=$ $0$-stable locus) of $\mf{M}_0(\bv,\bw)$ is nonempty, hence open and dense. Therefore, the simple locus of $\mf{M}_{\zeta}(\bv,\bw)$ is non-empty for all $\zeta\in \Theta$. Since this locus is contained in the $\zeta$-stable locus, the latter is also always non-empty. Now \cite[Theorem 1.15]{BellSchedQuiver} implies that the singular locus is precisely the strictly $\theta_0$-polystable locus. 

The morphism $\tau \colon \mf{M}_{\theta}(\bv,\bw) \to \mf{M}_{\theta_0}(\bv,\bw)$ is a surjective birational morphism because $\alpha \in \Sigma_0$. Moreover, since $\theta_0 \in \overline{C} \smallsetminus C$, there is a $\theta$-stable $\Pi$-module $M$ of dimension $\alpha$ that is not $\theta_0$-stable. Therefore, the image under $\tau$ of the corresponding point $[M]\in \mf{M}_\theta(\bv,\bw)$ is strictly $\theta_0$-polystable. The previous paragraph implies that $\tau[M]$ lies in the singular locus of $\mf{M}_{\theta_0}(\bv, \bw)$, so $\tau$ contracts at least one curve by Zariski's Main Theorem \cite[III, Corollary~11.4]{Hartshorne}. 
\end{proof}

\begin{thm} 
\label{thm:mainquiverVars}
 Let $\alpha\in \Sigma_0$. Then every GIT chamber $C$ satisfies Condition~\ref{cond:GIT}.
\end{thm}
\begin{proof}
Since $\alpha \in R^+$, Lemma~\ref{lem:genericfreeaction} implies that $G(\alpha)$ acts freely on $\mu^{-1}(0)^{\theta}$ for $\theta \in C$, so the quotient $\mf{M}_{\theta}(\bv,\bw)$ is nonsingular (see e.g \cite[Lemma~10.3]{CBConze}) and Theorem~\ref{thm:Kirwansurjectivequiver} says that $L_C$ is surjective. The origin in $\mu^{-1}(0)$ is a $G(\alpha)$-fixed point, so Lemma~\ref{lem:fixedpoint} says that $\overline{C}$ is strongly convex. For $\theta\in C$ and any $\theta_0\in \overline{C}\setminus C$ general in a wall, the morphism $\tau \colon \mf{M}_{\theta}(\bv,\bw) \to \mf{M}_{\theta_0}(\bv,\bw)$ contracts at least one curve by Lemma~\ref{lem:contractscurve}. Therefore $L_C$ is injective by Corollary~\ref{c:git-stronglyconvex}, so Condition~\ref{cond:GIT}(1) holds. In addition, the unstable locus in $\mf{M}_{\theta}(\bv,\bw)$ is precisely the preimage under $\tau$ of the singular locus of $\mf{M}_{\theta_0}(\bv,\bw)$; that is, the unstable locus is the exceptional locus of $\tau$, thanks to \cite[Theorem 1.15]{BellSchedQuiver} and the fact that the $\theta_0$-stable locus is nonempty (hence $\theta_0$-stable is equivalent to canonically $\theta_0$-polystable as in \emph{op.~cit.}). Since the smooth variety $\mf{M}_{\theta}(\bv,\bw)$ is a symplectic quotient, it is symplectic (see \cite[Section 8]{CBnormal}; note that the symplectic condition actually does not require smoothness, see  \cite[Theorem 1.2]{BellSchedQuiver}). As we have already established that $\tau$ is projective and birational, it is a symplectic resolution of singularities (with normal base, thanks to \cite{CBnormal}, although we don't require this for the argument). Therefore, it is semi-small \cite{Kaledinsympsingularities}, so 
\begin{equation}\label{eq:semismallquiverproof}
    \mr{codim}_{\mf{M}_{\theta}(\bv,\bw)} \, \mr{Uns}(\tau)%
    \ge \frac{1}{2} \mr{codim}_{\mf{M}_{\theta_0}(\bv,\bw)} \, \mr{Sing}\big(\mf{M}_{\theta_0}(\bv,\bw)\big).
\end{equation}
Note that this inequality also follows  explicitly for quiver varieties using, eg,   \cite[Corollary 6.4]{BellSchedQuiver}.

Since the singular locus of $\mf{M}_{\theta_0}(\bv,\bw)$ is a union of symplectic leaves, its codimension is even. There are two cases: 
\begin{enumerate}
\item If this codimension is at least $4$ then the codimension of the unstable locus is at least two. This depends only on $\theta_0$ in the GIT wall rather than on the chamber whose closure contains the wall, so this analysis applies equally to both morphisms $\tau_+$ and $\tau_-$ in diagram \eqref{eqn:flop}. Therefore, the GIT wall is a flipping wall. 
 \item Otherwise, the codimension of the singular locus of $\mf{M}_{\theta_0}(\bv,\bw)$ is two. Locally, the singularities of $\mf{M}_{\theta_0}(\bv,\bw)$ transverse to a codimension two leaf are Kleinian, which implies that the inequality in \eqref{eq:semismallquiverproof} is an equality. In other words, $\tau$ is divisorial. 
\end{enumerate}
 It remains to note that the GIT region $R_C$ is defined in such a way that the walls in the interior of $R_C$ cannot induce a divisorial contraction, so they are flipping by the above; boundary walls are not flipping, so they are divisorial by the above. Thus, Conditions~\ref{cond:GIT}(2) and (3) hold. 
\end{proof}

Our main result (Theorem~\ref{thm:movable}) therefore holds under the assumptions of Theorem~\ref{thm:mainquiverVars}, so we obtain Theorem~\ref{thm:mainquivertwointro}. In fact, Corollary \ref{c:partial-cr} implies the following stronger result:

\begin{cor}\label{cor:resgivenbyquiver} 
Let $\alpha\in \Sigma_0$, and let $C\subset \Theta$ be a chamber with GIT region $R_C$.
\begin{enumerate}
\item[\one] Projective partial crepant resolutions of $\mf{M}_0(\bv,\bw)$, taken
up to isomorphism over $\mf{M}_0(\bv,\bw)$, are in bijection with the GIT cones in $R_C$;
\item[\two] Under this bijection, the crepant resolutions of $\mf{M}_0(\bv,\bw)$ correspond to the GIT chambers.  
\end{enumerate}
\end{cor}

In particular, every projective partial crepant resolution of the affine quiver variety $\mf{M}_0(\bv,\bw)$ is of the form $f_\theta\colon \mf{M}_{\theta}(\bv,\bw) \to \mf{M}_0(\bv,\bw)$, for some $\theta \in R_C$. We also deduce from Corollary~\ref{cor:mds} the following result, independent of \cite{BCHM, NamikawaMDS}. 

\begin{cor}
 Let $\alpha\in \Sigma_0$. For any generic $\theta\in \Theta$, the quiver variety $\mf{M}_{\theta}(\bv,\bw)$ is a Mori Dream Space over $\mf{M}_0(\bv,\bw)$.
\end{cor}

\begin{rem}\label{r:BC-case}
    Actually, the assumption $\alpha \in \Sigma_0$ is stronger than what we need. It is enough for the proof above that: (a) for generic $\theta$, $\mf{M}_{\theta}(\bv,\bw) \to \mf{M}_0(\bv,\bw)$ is a resolution of singularities,  (b) $\mf{M}_{\theta_0}(\bv,\bw)$ is singular for nongeneric $\theta_0$, and (c) for $\theta_0$ in the interior of $R_C$, the $\theta_0$-stable locus in $V$ is nonempty.  These conditions are all satisfied in the setting of \cite{BellamyCraw} for $n > 1$, so Condition \ref{cond:GIT} applies in that situation, recovering their main result.  More generally, by \cite[Theorem~6.23, Remark~6.24]{SchedlerTirelli}, (a) is satisfied if and only if $\alpha$ is a ``flat root'', meaning that, for $\theta=0$, Definition \ref{def:Sigma0} holds except with a nonstrict inequality instead of a strict one.
     Under this assumption, condition (b) holds if and only if, whenever the proper decomposition in Definition 4.4 (for $\theta_0$) is into only two roots $\alpha= \beta^{(1)} + \beta^{(2)}$, either the inequality is strict, or for some other decomposition of $\alpha$ into positive roots lying in the $\mathbb{Q}$-linear span of $\beta^{(1)}$ and $\beta^{(2)}$, the inequality is strict.  
     Finally, assuming this, 
     condition (c) holds if and only if, for each decomposition into two roots $\alpha=\beta^{(1)}+\beta^{(2)}$, either (I) one has a strict inequality $p(\alpha)>p(\gamma^{(1)})+\cdots+p(\gamma^{(m)})$ for every decomposition into roots $\gamma^{(i)}$ in the  $\mathbb{Q}$-linear span of $\beta^{(1)}$ and $\beta^{(2)}$, or (II) for at least one such decomposition, $p(\alpha)=1+p(\gamma^{(1)})+\cdots+p(\gamma^{(m)})$.
     Geometrically, these conditions says that, for generic $\theta_0$ in the corresponding wall, either (I) $\alpha \in \Sigma_{\theta_0}$, or else (II) the quiver variety $\mf{M}_\theta(\bv,\bw)$ has singular locus of codimension two.
\end{rem}

\begin{rem}
One can consider quiver varieties associated to deformed preprojective algebras (at deformation parameter $\lambda$). The assumption $\lambda = 0$ is only required to deduce the surjectivity of $L_C$ from \cite{QuiverKirwan}. This is also the only place where we require $\alpha_{\infty} = 1$ (or $\bw \neq 0$). In forthcoming work,
 we will show that we can drop this condition, so that, for $\lambda=0$, $L_C$ is an isomorphism over $\Q$ for any $\alpha \in \Sigma_0$ except when $\alpha=2\beta$ for $(\beta,\beta)=-2$ (the O'Grady type singularity, for which $\mf{M}_\theta(\bv,\bw)$ is not terminal for generic $\theta$ and is resolved by blowing up the reduced singular locus).
\end{rem}

\subsection{The Namikawa--Weyl group}
 We now prove a general result about polyhedral cones and automorphisms of real vector spaces for which we could not find a suitable reference.
 
\begin{lem}\label{lem:integratransCCprime}
Let $C, C' \subset \R^n$ be the interiors of rational polyhedral cones such that $W:= \overline{C} \cap \overline{C'}$ is a common codimension-one face.
Let $\gamma$ be an integral automorphism of $\R^n$ with $\gamma(C) = C'$ fixing $W$ pointwise. Then $\gamma^2 = 1$ and $\mr{Fix}(\gamma)$ is the hyperplane spanned by $W$.   If $\overline{C}$ is strongly convex, then such an automorphism is unique.
\end{lem}

\begin{proof}
If $H \subset \R^n$ is the hyperplane spanned by the vectors in $W$ then $\gamma$ is the identity on $H$.  Write $H = \beta^{\perp}$ for some primitive vector $\beta \in \Z^n$. Then $\gamma(\beta) = -\beta + v$ for some $v \in H$, as $\gamma$ is an integral automorphism sending $C$ to $C'$.  As a result, $\det(\gamma)=-1$, and $\gamma$ must have an eigenvector $\beta'$ of eigenvalue $-1$. This proves the first assertion.

For the second assertion, suppose that $\gamma'$ is another integral automorphism fixing $W$ pointwise and sending $C$ to $C'$. Then $\varphi:=\gamma \circ \gamma'$ is an integral automorphism fixing $C$ and fixing $W$ pointwise. Thus, for $\beta$ as before, we have $\varphi(\beta) = \beta + u$ for some $u \in H$. Now if $\overline{C}$ is strongly convex, then for some other codimension-one face $W'$ of $\overline{C}$, we have that $u$ is not in the hyperplane spanned by $W'$. Therefore $\varphi(W')$ is either in the interior or the exterior of $\overline{C}$, which contradicts $\varphi(C)=C$. 
\end{proof}

For symplectic resolutions of conical symplectic singularities, such as $\mf{M}_{\theta}(\bv,\bw) \to \mf{M}_0(\bv,\bw)$, Namikawa \cite{Namikawa2} has shown that there is a finite Weyl group that acts (as a reflection group) on $H^2(\mf{M}_{\theta}(\bv,\bw),\Q)$. We refer to this action as the \textit{Namikawa--Weyl group}. 

\begin{prop}
\label{prop:Namikawa-Weyl}
 Let $\alpha\in \Sigma_0$. Each GIT region of the form $R_C$ is a 
simplicial cone. Reflections about the boundary walls of $R_C$ generate a group $\Gamma$ isomorphic to the Namikawa--Weyl group, which acts simply transitively on the set of all GIT regions. The union of these regions is all of $\Theta$.  Given GIT chambers $C, C'$, if $g \in \Gamma$ is the element satisfying $g(R_C)=R_{C'}$, then $L_{C'} = L_C \circ g$.
\end{prop}

\begin{proof}
Let $Y = \mf{M}_0(\bv,\bw)$ and $X_{\theta} = \mf{M}_{\theta}(\bv,\bw)$. 

Let $C$ be a GIT chamber with wall $W$ for which $\tau_-\colon \Xtheta{} \to X_{\theta_0}$, with $\theta \in C$ and a general $\theta_0 \in W$, is a divisorial contraction. If $C'$ is the other chamber with wall $W$ then $\tau_{+}$ from \eqref{eqn:flop} is also a divisorial contraction since $X_{\theta_0}$ has a codimension two leaf. Under the isomorphism $L_C$, the nef cone $L_C(\overline{C})$ contains $L_C(W)$. Since $\tau_-$ is a divisorial contraction, $W$ is a boundary wall of $R_C$. Then Theorem~\ref{thm:movable}\one \ says that $L_C(W)$ must be a boundary wall of the movable cone. Fix $\theta_0 \in W, \theta \in C, \theta'\in C'$ and let $\mc{O}(\theta_0)$ denote the corresponding polarising ample line bundle on $X_{\theta_0}$. Then \eqref{eqn:O1pullback} says that $L_C(\theta_0) = \tau_{-}^* (\mc{O}(\theta_0))$, but since $L_C(W)$ is a boundary wall of the movable cone of $X_{\theta}$, the latter is the unique minimal model of $Y$ dominating $X_{\theta_0}$. This implies that $X_{\theta} \cong X_{\theta'}$ over $Y$. We deduce that $L_C(C) = L_{C'}(C')$ is the ample cone of $X_{\theta}$ over $Y$. Moreover, 
\[
L_C(\theta_0) = \tau_{-}^*(\mc{O}(\theta_0)) = \tau^*_{+}(\mc{O}(\theta_0)) = L_{C'}(\theta_0)
\]
shows that $L_C |_{W} = L_{C'} |_{W}$. Hence, Theorem~\ref{thm:Kirwansurjectivequiver} says that $\gamma := L_{C'}^{-1} \circ L_C$ is obtained from an integral automorphism, $\ell_{C'}^{-1} \circ \ell_C$, of $\Theta_{\mathbb{Z}} := \{\theta \in \Hom(\ZZ^{Q_0}, \ZZ) \mid \theta(\alpha)=0\}$,
and it  maps $C$ to $C'$ and fixes $W$ pointwise. Lemma~\ref{lem:integratransCCprime} says that $\gamma^2 = 1$.  Moreover,
 $L_{C'}(R_{C'})=L_C(R_C) = \Mov(\Xtheta{}/Y)$ implies that $R_{C'}=\gamma(R_C)$.

Next, inside $N^1(X_{\theta}/Y)$, \cite[Proposition 2.17]{BLPWAst} shows that $\Mov(\Xtheta{}/Y)$ is a fundamental domain for the Namikawa--Weyl group, generated by reflections about the boundary walls of $\Mov(\Xtheta{}/Y)$.  Since the Namikawa--Weyl group is a Weyl group \cite{Namikawa2} acting on the reflection representation, its fundamental regions are simplicial cones. Pulling this back via $L_C$ shows that $R_C$ is a simplicial cone. If $s$ is the reflection in the Namikawa--Weyl group about the wall $L_C(W)$ then the uniqueness statement of Lemma~\ref{lem:integratransCCprime} implies that $L_C \circ \gamma = s \circ L_C$. Hence the reflections about the boundary walls of $R_C$ generate a group $\Gamma \subset \GL(\Theta)$ isomorphic to the Namikawa--Weyl group. This group acts with $R_C$ as a fundamental region. In particular, $\Theta = \bigcup_{g \in \Gamma} g(R_C)$. 

We claim $g(R_C) = R_{C'}$ for some GIT chamber $C'$. This is done by induction on the length of $g$, the case $\ell(g) = 1$ having been done already. If $g = \gamma h$ with $\ell(h) < \ell(g)$ then $h(R_C) = R_{h(C)}$ and applying the previous argument with $C$ replaced by $h(C)$ shows that $g(R_C) = \gamma(R_{h(C)}) = R_{C'}$ for some $C'$. A similar induction shows that if $C'' \subset R_{C'}$ and $g(R_C) = R_{C'}$ then $L_{C''} \circ g = L_C$. 
\end{proof}

\subsection{Combinatorics and hyperplane arrangements}\label{s:combinatorics}
The results of this final section 
require that the dimension vector $\alpha$ for the quiver $Q$ is indivisible, but they 
do not require $\alpha$ to come from a nonzero framing. Thus, we write $\mf{M}_\theta(\alpha)=\mu^{-1}(0)^\theta\git \; G(\alpha)$ for the quiver variety. 
Note that the indivisibility assumption on $\alpha$ ensures that $\mf{M}_\theta(\alpha)$ is nonsingular for general $\theta$. 

For any root $\gamma\in R$, consider the hyperplane $\gamma^\perp:= \{\theta\in \Hom(\Z^{Q_0},\Q) \mid \theta(\gamma)=0\}$. Note that $\alpha^\perp = \Theta$. 

\begin{defn}
\label{defn:quiverGITarrangement}
	Consider the hyperplane arrangement in $\Theta$ given by
	$$
	\mathcal{A}_\alpha = \big\{\beta^\perp \cap \alpha^\perp \mid  \textrm{$\alpha = \beta + (\alpha - \beta)$ is a decomposition into two roots in $R^+$} \big\}.
	$$
\end{defn}

The hyperplane arrangement $\mathcal{A}_\alpha$ determines a polyhedral wall-and-chamber decomposition of $\Theta$, and the resulting (closed) cones form a complete fan in $\Theta$. The interior of each top-dimensional cone in the fan of $\mathcal{A}_\alpha$ is the intersection of $\Theta$ with a connected component of the locus
 \[
 \Theta_{\mathbb{R}} \setminus \bigcup_{\gamma^\perp\in \mathcal{A}_\alpha} \gamma^\perp.
 \]
 
The goal of this section is to show that the above fan is precisely the GIT fan. The following result generalises \cite[Corollary~4.7]{WuQuiver}. 

\begin{lem}
\label{lem:rootsum}
	If $ \beta,\gamma \in R^+$ and $\langle \beta,\gamma \rangle < 0$, then $\beta+\gamma \in R^+$ as well.
\end{lem}

\begin{proof}
  If either $\beta$ or $\gamma$ is real then this is true thanks to \cite[Proposition 3.6(a)]{KacBook} (see also Proposition 5.1(c) of \emph{op.~cit.}): if $\beta$ is a real root then the restriction of an integral representation to $\mathfrak{g}(\beta) \cong \mathfrak{sl}_2$ is a sum of finite-dimensional modules, and the adjoint representation is integrable. We give a purely combinatorial proof of the general case.

  Let $\langle \beta,\gamma \rangle = -m < 0$.  If $\gamma$ is real and $m=-1$, then $\beta+\gamma$ is a reflection of $\beta$, so also a root; the same is true swapping $\beta$ and $\gamma$. We may assume therefore that either $m\neq -1$ or $\beta,\gamma$ are both imaginary.

  Let $\eta^{(0)} := \beta+\gamma, \; \beta^{(0)} := \beta, \;\gamma^{(0)} := \gamma$. Inductively, let us apply a maximal sequence of simple reflections so that $\eta^{(j)} = s_{i_j}(\eta^{(j-1)})$, with $\eta^{(j-1)} < \eta^{(j)}$; this means that $\langle \eta^{(j-1)}, e_{i_j} \rangle > 0$.  Let $\beta^{(j)} := s_{i_j}(\beta^{(j-1)})$ and $\gamma^{(j)} := s_{i_j}( \gamma^{(j-1)})$.  We claim that under this sequence $\beta^{(j)}, \gamma^{(j)}$ always remain positive. If, at some stage, $\beta^{(j)}$ is negative, then $\beta^{(j-1)}=e_{i_j}$. Then $-m= \langle \beta^{(j-1)}, \gamma^{(j-1)} \rangle = \langle e_{i_j}, \gamma^{(j-1)} \rangle$. But $\langle \eta^{(j-1)}, e_{i_j} \rangle \geq 1$ by assumption, so $\langle e_{i_j}, \gamma^{(j-1)} \rangle \geq -1$. Thus $m=1$. In this case $\beta$ and $\gamma$ are both imaginary, which contradicts $\beta^{(j-1)}=e_{i_j}$.

Since $\beta^{(j)}$ and $\gamma^{(j)}$ are always positive roots with nonzero pairing, their sum is always connected and positive.  So $\eta^{(j)}$ remains connected and positive. Eventually, this sequence must terminate (say at $\eta^{(k)}$). Then $\langle \eta^{(k)},e_i\rangle \le 0$ for all loop free vertices $i$, implying that $\eta^{(k)}$ is in the fundamental domain. This implies that $\beta+\gamma$ is an imaginary root.
\end{proof} 

Given a tuple $D := (\alpha^{(1)}, \ds, \alpha^{(m)})$ of roots in $R^+$ we associate the quiver $Q_D$ whose vertices are $1,\ds,m$, with $-\langle \alpha^{(i)}, \alpha^{(j)} \rangle$ arrows from $i$ to $j$ for $i \neq j$. Given a decomposition $\alpha = \alpha^{(1)} + \cdots + \alpha^{(m)}$, we associate this tuple and hence the quiver.

\begin{lem}
\label{lem:QDconnected}
	Suppose that $\alpha \in \Sigma_{\theta}$ has a decomposition $D : \alpha = \alpha^{(1)} + \cdots + \alpha^{(m)}$ into roots in $R^+_{\theta}$. Then the associated quiver $Q_D$ is connected.
\end{lem}

\begin{proof}
 If $\{1,...,m\} = I \cup J$ with $I$ and $J$ disconnected from each other in $Q_D$, then we get a decomposition $\alpha = \alpha_I + \alpha_J$, where $\langle \alpha_I, \alpha_J \rangle = 0$. This implies that $p(\alpha) < p(\alpha_I) + p(\alpha_J)$. Taking canonical decompositions of $\alpha_I$ and $\alpha_J$, and applying \cite[Lemma~7.3]{BellSchedQuiver}, we get a contradiction to the fact that $\alpha \in \Sigma_{\theta}$.
\end{proof}

\begin{prop}\label{prop:compprop3}
For every decomposition $D: \alpha = \alpha^{(1)} + \cdots + \alpha^{(m)}$, with connected quiver $Q_D$, the intersection $\cap_i (\alpha^{(i)})^\perp$ equals an intersection of hyperplanes in $\mc{A}_\alpha$.
\end{prop}

The proof of this is based on an easy, purely combinatorial statement:

\begin{lem}
\label{lem:SQ}
Let $Q$ be a connected (undirected) graph with vertex set $Q_0$. For $J \subseteq Q_0$ define $e_J := \sum_{j \in J} e_j$, where $e_j\in \ZZ Q_0$ is the trivial path at vertex $j$. Then $\ZZ Q_0$ is spanned by the set
$$
S_Q := \{e_J \mid J \subseteq Q_0 \text{  
 is such that }J \text{ and }Q_0\setminus J 
 \text{ are connected} \}.
$$
\end{lem}

\begin{proof}
	By induction on $|Q_0|$. Note that in a connected graph there is always a vertex $j \in Q_0$ which can be removed leaving a connected graph (this is obvious for a tree, and every connected graph has a spanning tree).  Then $e_j$ is in the set $S_Q$ above. Let $j_0 \in Q_0$ be such a vertex.  Define $Q'$ to be the graph obtained from $Q$ by deleting the vertex $j_0$ and all edges that have an endpoint at $j_0$. By induction $\ZZ Q'_0$ is spanned by $S_{Q'}$. But for each $J \subseteq Q_0'$ such that both $J$ and $Q_0'\smallsetminus J$ are connected, either $J \cup \{j_0\}$ or $Q_0 \setminus J$ is connected, so $e_J$ or $e_J+e_{j_0}$ is in $S_{Q'}$. Thus $\Span(S_Q)$ contains $\ZZ \cdot e_{j_0}$, while the quotient $\Span(S_Q)/\ZZ  \cdot e_{j_0}$ contains $\Span(S_{Q'}) = \ZZ Q'_0$, so $S_Q$ spans $\ZZ Q_0$.
\end{proof}

\begin{proof}[Proof of Proposition~\ref{prop:compprop3}]
    Lemma~\ref{lem:rootsum} implies that for every $J \subset (Q_D)_0$ connected, the sum $\beta_J := \sum_{j\in J} \alpha^{(j)}$ is in $R^+$. In the case $e_J \in S_{Q_D}$ as in Lemma~\ref{lem:SQ}, we get that both $\beta_J$ and $\alpha - \beta_J$ belong to $R^+$. Thus, $\beta_J$ is the perpendicular vector to a hyperplane in $\mc{A}_\alpha$. Lemma \ref{lem:SQ}, applied to $Q_D$, then says that the intersection of these hyperplanes $\beta_J^{\perp}$, for $e_J \in S_{Q_D}$, equals the intersection of the hyperplanes $(\alpha^{(i)})^\perp$ for $1\leq i\leq m$, since intersecting hyperplanes produces the linear subspace perpendicular to the span of the normal vectors. 
\end{proof}

\begin{thm}\label{thm:GITquiverarrangement}
 Assume $\alpha \in \Sigma_0$ is indivisible. The GIT fan equals the fan given by the arrangement $\mc{A}_{\alpha}$.  
\end{thm}

\begin{proof}
It suffices to show that the GIT walls
are precisely the union of the hyperplanes in $\mc{A}_{\alpha}$. In other words, $\mf{M}_{\theta}(\alpha)^s = \mf{M}_{\theta}(\alpha)$ if and only if $\theta$ lies in the complement to the hyperplanes in $\mc{A}_{\alpha}$. 

Assume that $\theta$ is a general element of $\gamma^{\perp} \in \mc{A}_{\alpha}$. Then there exists a positive root $\beta$ such that $\alpha - \beta \in R^+$ and $\theta(\beta) = 0$. It is a consequence of \cite[Theorem~1.3]{BellSchedQuiver} that there exists a $\theta$-polystable representation of dimension vector $\eta$ for any $\eta \in \N R^+_{\theta}$. In particular, this implies that there exist $\theta$-polystable representations $M,N$ of dimension vector $\beta$ and $\alpha-\beta$ respectively. The point $[M \oplus N] \in \mf{M}_{\theta_0}(\alpha)$ is strictly $\theta$-polystable. Hence $\gamma^{\perp}$ is a GIT wall.

Conversely, if $\theta \in \Theta$ lies on some GIT wall then, by definition, there exists a properly $\theta$-polystable representation $M = M_1^{\oplus n_1} \oplus \cdots \oplus M_k^{\oplus n_k}$ with $\alpha^{(i)} := \dim M_i$ belonging to $\Sigma_{\theta}$. Counting the $\alpha^{(i)}$ with multiplicity gives a decomposition $D$ of $\alpha$. Since $\alpha \in \Sigma_0$, the associated quiver $Q_D$ is connected by Lemma~\ref{lem:QDconnected}. Then Proposition~\ref{prop:compprop3} implies that $\theta$ lies on some hyperplane in $\mc{A}_\alpha$.
 \end{proof}

We note the following useful consequence of the proof of Theorem~\ref{thm:GITquiverarrangement}. 

\begin{cor}
If $\alpha \in \Sigma_0$ is indivisible, then the quiver variety $\mf{M}_{\theta}(\alpha)$ is nonsingular if and only if $\theta$ does not lie on any hyperplane in $\mc{A}_{\alpha}$. 
\end{cor}

We may describe the GIT regions $R_C$ more explicitly. By \cite[Theorem 1.20]{BellSchedQuiver}, the walls in the boundary of $R_C$ all lie in the hyperplanes $\beta^\perp$ where $\beta$ is a  \emph{codimension two root}, meaning that there is a codimension two stratum $\mf{M}_{\theta}(\alpha)_{\tau}$ where the dimension vectors $\beta^{(i)}$ appearing in the representation type $\tau$ are all rational combinations of $\alpha$ and $\beta$.  In fact, in \cite{BelSchMinimalDegenerations},  we will show that this is equivalent to the condition that $\beta$ and $\alpha-\beta$ are both roots, and $(\beta, \alpha-\beta)=-2$ (but we 
do not need this fact here). Then $R_C$ is the closure of one of the complementary regions of these hyperplanes, namely, the one containing $C$. Conversely every such region can be used as $R_C$, and $C$ can be taken to be any GIT chamber inside it.
  \section{Hypertoric varieties}\label{sec:hypertoric}
  We now show that our main results also apply to the class of nonsingular hypertoric varieties, leading to a proof of Theorem~\ref{thm:conjecturefortori}. We work over the complex numbers in this section.

\subsection{GIT construction}
 Hypertoric varieties were originally constructed as hyperk\"ahler quotients by Bielawski and Dancer~\cite{BielawskiDancer}, where they were called toric hyperk\"ahler manifolds (the name ``hypertoric'' for the possibly singular algebraic varieties was coined later in work of Harada and Proudfoot, noting that they are not toric varieties.) Here we recall their construction as holomorphic symplectic varieties by GIT following Hausel and Sturmfels~\cite{HS} (see also Konno~\cite{Konno03}).
 
 For $n, r\in \mathbb{N}$ with $r<n$, consider the action of the algebraic torus $G:=(\CC^\times)^r$ on the complex symplectic vector space $T^*\CC^n = \CC^n\times (\CC^n)^*$, where the matrix that records the weights of the action is of the form $(A,-A)$, where $A$ is an $r\times n$ integer-valued matrix whose columns $a_1, \dots, a_n$ span $\ZZ^r$. Note that this forces the $r \times r$-minors of $A$ to be relatively prime. The $G$-action is Hamiltonian for the natural symplectic structure on $T^*\CC^n$, and the induced moment map $\mu\colon T^*\CC^n\to \mathfrak{g}^*$ satisfies 
 \[
 \mu(z,w)=\sum_{i=1}^n z_iw_i \cdot a_i.
 \]
 Choose an integer $n \times (n-r)$ matrix $B$ forming the short exact sequence
 \[
 0 \to \Z^{n-r} \stackrel{B}{\longrightarrow} \Z^n \stackrel{A}{\longrightarrow} \Z^r \to 0. 
 \]
 If no row of the matrix $B$ is zero (equivalently, when the torus $G$ contains no dilations along a single axis) then the locus $\mu^{-1}(0)$ is an affine variety by \cite[Lemma~4.7]{HypertoricBB}, and for any character $\theta\in G^\vee$, the corresponding \emph{hypertoric variety} is defined to be
 \[
 \Xtheta{}:= \mu^{-1}(0)\git_\theta \, G.
 \]
 Recall that a matrix $A$ is said to be \textit{unimodular} if all of the non-zero $r \times r$ minors of $A$ belong to $\{ -1,0, 1 \}$ (equivalently, the $(n-r) \times (n-r)$-minors of $B$ belong to $\{ -1,0, 1 \}$). Under the assumption that no row of the matrix $B$ is zero, it is shown in \cite[Proposition~6.2]{HS} that $\Xtheta{}$ is nonsingular for general $\theta$ if and only if $A$ is unimodular. Note in addition, that the interior of every top-dimensional cone in the GIT fan is a chamber, because $G$ is a torus \cite[Corollary~4.1.10]{DolgachevHu98}.
 
\subsection{Applying the main result}

 In order to apply our Theorem~\ref{thm:movable} to nonsingular hypertoric varieties, we show that Condition~\ref{cond:GIT} holds. Much of the heavy lifting was done by Konno~\cite{Konno03}.

\begin{thm}\label{thm:hypertoric}
Assume that $A$ is unimodular and no row of the matrix $B$ is zero, so the hypertoric variety $\Xtheta{}$ is nonsingular for general $\theta$. Then Condition~\ref{cond:GIT} holds for every chamber $C$, and hence Theorem~\ref{thm:movable} applies. In particular, every projective crepant resolution of $X_0$ is a hypertoric variety $\Xtheta{}$ for some generic $\theta$.
\end{thm}

\begin{proof}
Let $C$ be a chamber and $\theta \in C$. Since $X_{\theta}$ is nonsingular, \cite[Theorem~1.1]{SVdBhypertoric} shows that there exist $\eta_1, \ds, \eta_k \in \Theta$ such that $\mc{T} = \bigoplus_{i = 1}^k L_C(\eta_i)$ is a tilting bundle on $X_{\theta}$. This implies that the line bundles $L_C(\eta_i)$ for $1\leq i\leq k$ span the Grothendieck group $K_0(X_{\theta})$. Since $\det \colon K_0(X_{\theta}) \to \mr{Pic}(X_{\theta})$ is surjective, we deduce that $L_C$ is surjective. We note that the closed cone $\overline{C}$ is strongly convex by Lemma~\ref{lem:fixedpoint} because $\mu^{-1}(0)$ has a $G$-fixed point. Thus, Corollary~\ref{c:useful-criterion}(1) holds.  

In order to apply results from \cite{Konno03,Konnohypertoric}, we note that condition (C1) from \cite{Konno03} 
is equivalent to $A$ being unimodular and condition (C2) is equivalent to no row of $B$ being zero. Assume now that chambers $C_{\pm}$ share a wall in $R_C$, and let $\theta_\pm \in C_\pm$ and choose $\theta_0$ be general in this wall. Let $S \subset X_{\theta_0}$ denote the set of points whose corresponding closed $G$-orbit in $\mu^{-1}(0)^{\theta_0}$ consists of points with stabiliser of positive dimension. Then Konno 
says that $\tau_{\pm} \colon X_{\theta_{\pm}} \setminus \mr{Uns}(\tau_{\pm}) \to X_{\theta_0} \setminus S$ is an isomorphism (\cite[proof of Theorem~6.3 on page 306]{Konno03} or \cite[Theorem~6.4(4)]{Konnohypertoric}) and $\tau_{\pm} |_{\mr{Uns}(\tau_{\pm})}$ is a $\mathbb{P}^r$-bundle over $S$ (\cite[Lemma~6.8(1)]{Konno03} or \cite[Theorem~6.4(3)]{Konnohypertoric}). Crucially, \cite[Lemma~6.8(3)]{Konno03} says that $r \ge 1$ and hence $\tau_{\pm}$ always contracts a curve. Since $\tau_{\pm}$ is a Poisson morphism (with $X_{\theta_{\pm}}$ symplectic) and $\tau_{\pm}$ is not an isomorphism over the unstable locus, the image of the unstable locus must equal the singular locus of $X_{\theta_0}$. Thus, Corollary~\ref{c:useful-criterion}(2) holds. 

Finally, we note that a wall is small if $r > 1$, otherwise it is a divisorial boundary wall. In particular, Corollary~\ref{c:useful-criterion}(3) holds. We deduce from Corollary~\ref{c:useful-criterion} that Condition~\ref{cond:GIT} holds.   
\end{proof}

\begin{proof}[Proof of Theorem~\ref{thm:conjecturefortori}]
This is immediate from Theorem~\ref{thm:hypertoric}.
\end{proof}

\section{Crepant resolutions of some threefold quotient singularities}
 We now show that our main results apply to projective, crepant resolutions of certain Gorenstein, threefold quotient singularities, including all polyhedral singularities. That our methods can be applied to threefolds emphasises the fact that our results do not in any way rely on the holomorphic symplectic structure of Nakajima quiver varieties.
 
\subsection{McKay quiver moduli spaces}
 Let $\Gamma\subset \SL(3,\kk)$ be a finite subgroup. The affine quotient singularity $\mathbb{A}^3/\Gamma:=\Spec \kk[\mathbb{A}^3]^\Gamma$ is a normal, Gorenstein threefold that admits a projective, crepant resolution. Rather than recall the construction of  Bridgeland, King and Reid~\cite{BKR}, it is convenient for our purpose to recall the more general construction appearing in~\cite[Section~2]{CrawIshii}.
 
  Let $\Irr(\Gamma)$ denote the set of isomorphism classes of irreducible representations of $\Gamma$, and write $R(\Gamma)=\bigoplus_{\rho\in \Irr(\Gamma)} \ZZ \rho$ for the representation ring of $\Gamma$. A \emph{$\Gamma$-constellation} is a $\Gamma$-equivariant coherent sheaf $F$ on $\mathbb{A}^3$ such that $H^0(F)$ is isomorphic to the regular representation $R=\bigoplus_{\rho\in \Irr(\Gamma)} R_\rho \otimes \rho$ as a $\kk[\Gamma]$-module. Note that $H^0(F)$ is a module over the skew group algebra $\kk[\mathbb{A}^3]\rtimes \Gamma$ of dimension vector $(\dim \rho)_{\rho \in \Irr(\Gamma)}$, and conversely, the sheaf on $\mathbb{A}^3$ associated to any such module is a $\Gamma$-constellation. Consider the rational vector space
  \[
  \Theta:=\big\{\theta\in \Hom_\ZZ(R(\Gamma),\QQ) \mid \theta(R)=0\big\}. 
  \]
  For $\theta\in \Theta$, a $\Gamma$-constellation $F$ is \emph{$\theta$-semistable} if every proper nonzero $\Gamma$-equivariant coherent subsheaf $F'$ of $F$ satisfies $\theta(F'):=\theta(H^0(F')) \geq 0$; it is \emph{$\theta$-stable} if these inequalities are strict. Two $\theta$-semistable $\Gamma$-constellations are S-equivalent if their composition series agree in the abelian category of $\theta$-semistable $\Gamma$-constellations. The space $\Theta$ supports a polyhedral fan characterised by the following property: $\theta\in \Theta$ lies in the interior of a top-dimensional cone if and only if every $\theta$-semistable $\Gamma$-constellation is $\theta$-stable, in which case we say $\theta$ is \emph{generic} \cite[Lemma~3.1]{CrawIshii}. In particular, the interior of every top-dimensional cone in the GIT fan is a \emph{chamber}, so a \emph{wall} is a codimension-one face of the closure of any chamber. 
  
  Let $W$ denote the given three-dimensional representation of $\Gamma$, and consider the affine scheme $\{B\in \Hom_{\kk[\Gamma]}(R,W\otimes R) \mid B\wedge B=0\}$ parametrising $\Gamma$-constellations. Let $V$ denote the irreducible component of this scheme containing the free $\Gamma$-orbits. Isomorphism classes of $\Gamma$-constellations in $V$ correspond to orbits in $V$ under the action of  $G_\Gamma=\prod_{\rho\in \Irr(\Gamma)} \GL(\dim \rho)$ by change of basis on the summands of $R$. For any integer-valued $\theta\in \Theta$, consider the character $\chi_\theta\in {G_\Gamma}^\vee$ satisfying $\chi_\theta(g) = \prod_{\rho\in \Irr(\Gamma)} \det(g)^{\theta(\rho)}$ for $g\in G_\Gamma$. As in the construction by King~\cite{KingStable}, the GIT quotient 
  \[
 \mathcal{M}_\theta:= V \git_{\chi_\theta} G_\Gamma
 \]
 is the coarse moduli space of S-equivalence classes of $\theta$-semistable $\Gamma$-constellations that are deformations of a free $\Gamma$-orbit. The dimension vector $(\dim \rho)_{\rho\in \Irr(\Gamma)}$ is indivisible, so for any generic $\theta\in \Theta$, the GIT quotient $\mathcal{M}_\theta$ is the fine moduli space of $\Gamma$-constellations (that are deformations of a free $\Gamma$-orbit) up to isomorphism. 

 The tautological family on $\mathcal{M}_\theta$ is a locally-free sheaf $\mathcal{R}=\bigoplus_{\rho\in \Irr(\Gamma)} \mathcal{R}_\rho \otimes \rho$ and a tautological $\Gamma$-equivariant homomorphism $\mathcal{R}\to W\otimes \mathcal{R}$, where $\mathcal{R}_\rho$ has rank $\dim(\rho)$. We normalise the family so that the summand indexed by the trivial representation is the trivial bundle; see \cite[Section~2]{CrawIshii} for details. 
 
 \begin{prop}
 \label{prop:BKR}
 Let $C\subseteq \Theta$ be a chamber and let $\theta\in C$. Then 
 \begin{enumerate}
     \item[\one] variation of GIT quotient given by sending $\theta\rightsquigarrow 0$ induces a projective crepant resolution $f_\theta\colon \mathcal{M}_\theta\to \mathbb{A}^3/\Gamma$ that sends each $\Gamma$-constellation to its supporting $\Gamma$-orbit; and
     \item[\two] the linearisation map $L_C$ is surjective.
     \end{enumerate}
 \end{prop}
 \begin{proof}
 Part~\one\ is due to \cite{BKR}, though it appears in this form only in \cite[Proposition~2.2, Theorem~2.5]{CrawIshii}; note that the singularity $\mathbb{A}^3/\Gamma\cong\mathcal{M}_0$ is only an irreducible component of the affine quotient $\{B\in \Hom_{\kk[\Gamma]}(R,W\otimes R) \mid B\wedge B=0\}\git\; G$ in general. Part \two\ appears in \cite[Section~3.2]{CrawIshii},  or more explicitly, as \cite[Corollary~3.9]{CrawGaledual}.
 \end{proof}

  \subsection{The linearisation map}
  Our interest lies with those quotient singularities for which the linearisation map is an isomorphism. This property can be characterised in several ways as follows.
      
 \begin{lem}
 \label{lem:nosenior}
 Let $\Gamma\subset \SL(3,\kk)$ be a finite subgroup. The following statements are equivalent:
\begin{enumerate}
    \item[\one] every nontrivial conjugacy class of $\Gamma$ is `junior' in the sense of Ito and Reid~\cite{ItoReid96};
    \item[\two] some (and hence any) projective crepant resolution $f\colon X\to \mathbb{A}^3/\Gamma$ has all fibres of dimension at most one;
    \item[\three] for any GIT chamber $C\subset \Theta$ and $\theta\in C$, the moduli space $\mathcal{M}_\theta$ contains no proper surfaces;
    \item[\four] for any GIT chamber $C\subset \Theta$, the linearisation map $L_C$ is an isomorphism.
\end{enumerate}
 \end{lem}
 \begin{proof}
 Since $\mathbb{A}^3/\Gamma$ admits a projective crepant resolution $f\colon X\to \mathbb{A}^3/\Gamma$, condition \one\ is equivalent by \cite[Theorem~1.6]{ItoReid96} to the statement that $X$ contains no proper $f$-exceptional prime divisors, which is equivalent to $f$ having all fibres of dimension at most one. This holds for one crepant resolution if and only if it holds for all such \cite[Corollary~3.54]{KollarMori}, so \one\ and \two\ are equivalent. For any chamber $C\subset \Theta$ and $\theta\in C$, the morphism $f_\theta\colon \mathcal{M}_\theta\to \mathbb{A}^3/\Gamma$ is a projective crepant resolution by Proposition~\ref{prop:BKR}, so \two\ is equivalent to \three. Finally, 
 \cite[Lemma~4.2]{CrawGaledual} shows that the kernel of $L_C$ is dual to a vector space spanned by the numerical classes of proper surfaces in $\mathcal{M}_\theta$. Thus, there are no such surfaces if and only if $\ker(L_C)=0$. The result follows from Proposition~\ref{prop:BKR}\two.
 \end{proof}
  
 \begin{example}
 A simple and much-studied example is that of the subgroup $\Gamma=\ZZ_2\times \ZZ_2$ in $\SL(3,\kk)$ generated by the diagonal matrices $\text{diag}(1,-1,-1)$ and $\text{diag}(-1,-1,1)$. The toric threefold $\mathbb{A}^3/\Gamma$ admits four projective, crepant toric resolutions, one of which has exceptional locus comprising three $(-1,-1)$-curves meeting at a point; the remaining three such resolutions are obtained by flopping one of these curves. All four of these resolutions can be constructed as fine moduli spaces of $\theta$-stable $\Gamma$-constellations for some generic $\theta$; see \cite[Chapter~5]{Crawthesis} or \cite[Example~3.4, Remark~7.5]{Wemyss18}.   
 \end{example}
 
 \begin{example}[\textbf{Polyhedral singularities}]
  \label{exa:polyhedral} 
  It is classical that every finite subgroup $\Gamma\subset \SO(3,\mathbb{R})$ is a cyclic group, a dihedral group, or the rotational symmetry group of either the tetrahedron, the octahedron or the icosahedron.  The quotient singularity $\mathbb{A}^3/\Gamma$ is called \emph{polyhedral singularity}.  For each of these groups, Gomi, Nakamura and Shinoda~\cite{GNS00,GNS04} showed that the Hilbert--Chow morphism for the $\Gamma$-Hilbert scheme (this is $f_\theta$ from Proposition~\ref{prop:BKR}\one\ for $\theta$ as in \cite[Proposition~5.9]{Crawthesis})  satisfies the condition from Lemma~\ref{lem:nosenior}\two. The crepant resolution is unique when $\Gamma$ is cyclic, but for the dihedral and tetrahedral cases, Nolla de Celis and Sekiya~\cite{NollaSekiya17} subsequently proved that every projective, crepant resolution of $\mathbb{A}^3/\Gamma$ is of the form $\mathcal{M}_\theta$ for some generic $\theta$. Compare Remark~\ref{rem:Wemyss}.
 \end{example}
 
 \subsection{On GIT walls}
 \label{sec:threefoldsGITwalls}
 We now turn our attention to the GIT walls in $\Theta$. For adjacent chambers $C_+, C_-$ separated by a wall, variation of GIT quotient induces morphisms $\tau_\pm \colon \mathcal{M}_{\theta_{\pm}} \to \mathcal{M}_{\theta_0}$  of schemes over $Y=\mathbb{A}^3/\Gamma$ as in \eqref{eqn:flop}.
 
 The proof of the next result builds on the proof of \cite[Proposition~4.4]{CrawIshii}.
 
 \begin{lem}
 \label{lem:unstablesurfacescontracted}
 Let $\Gamma$ satisfy the equivalent conditions from Lemma~\ref{lem:nosenior}.  Suppose that $\Uns(\tau_-)$ has an irreducible component $D$ of codimension one. Then $D$ is contracted by $\tau_-$ onto a curve.
  \end{lem}
 \begin{proof}
By Lemma~\ref{lem:nosenior}, the divisor $D$ cannot be contracted to a point. We claim that the resolution $f\colon \mathcal{M}_{\theta_-}\to \mathbb{A}^3/\Gamma$ that sends each $\theta_-$-stable $\Gamma$-constellation to its supporting $\Gamma$-orbit contracts $D$ onto a curve. It suffices to show that $f(D)$ is contained in a curve. The union of the fixed loci in $\mathbb{A}^3$ under all nontrivial elements of $\Gamma$ is a finite union of lines through the origin, so its image in $\mathbb{A}^3/\Gamma$ is a curve $Z$.  Any point in $\mathcal{M}_{\theta_-}$ lying over the complement of $Z$ corresponds to a simple $\Gamma$-constellation because $\Gamma$ acts freely on the corresponding locus of $\mathbb{A}^3$, so it is $\theta_0$-stable. But the family of $\Gamma$-constellations over $D$ is not $\theta_0$-stable, so $\ell:=f(D)$ is contained in $Z$ as required.

Let $\pi \colon \mathbb{A}^3 \to \mathbb{A}^3/\Gamma$ be the quotient map and consider a nonzero $x\in \pi^{-1}(\ell)$. Then $x\in \mathbb{A}^3$  has a non-trivial stabiliser $\Gamma^\prime$. As in the proof of \cite[Lemma~8.1]{BKR}, the restriction functor provides an equivalence from the category of $\Gamma$-constellations supported on the orbit $\Gamma\cdot x$ to the category of $\Gamma^\prime$-constellations supported at $x\in \mathbb{A}^3$, and moreover, the restriction map that sends a character $\chi_\theta$ of $G_{\Gamma}$ to the character $\chi_{\theta^\prime}:=\res^{G_\Gamma}_{G_{\Gamma'}}(\chi_\theta)$ of $G_{\Gamma^\prime}$ determines the $\QQ$-linear map $\Theta \to \Theta_{\Gamma^\prime}$ between the spaces of  stability parameters for $\Gamma$- and $\Gamma^\prime$-constellations. This compatibility implies in particular that the restriction of a $\theta$-stable $\Gamma$-constellation supported on $\Gamma\cdot x$ is a $\theta^\prime$-stable $\Gamma^\prime$-constellation supported on $x$. Thus, if we write $f_{\theta^\prime}\colon \mathcal{M}_{\theta^\prime}(\Gamma^\prime)\to \mathbb{A}^3/\Gamma^\prime$ for the morphism sending each $\Gamma^\prime$-constellation to its supporting $\Gamma^\prime$-orbit, and $E_\ell:= (f_{\theta^\prime})^{-1}(\ell)$ for the preimage of
  $\ell$, then the restriction functor identifies $f_{\theta_-}\vert_D \colon D\to \ell$ with $f_{\theta^\prime}\vert_{E_\ell}\colon E_\ell\to \ell$. 
 
 This description of $f_{\theta_-}\vert _D$ allows us to study $\tau_-\vert_D$. Indeed, the action of $\Gamma^\prime$ fixes $x\in \mathbb{A}^3\smallsetminus \{0\}$, so we may choose coordinates with $\Gamma^\prime\subset \SL(2,\kk)\times \id\subset \SL(3,\kk)$. Since $f_{\theta^\prime}$ is a crepant resolution, we have that $f_{\theta^\prime}=f\times \id_{\mathbb{A}^1}$ where $f$ is the minimal resolution of an ADE singularity. The morphism $f_{\theta_-}$ is obtained by varying the stability parameter to zero, so $f_{\theta_-}$ factors via $\tau_-$. The restriction functor identifies $\tau_-\vert_D$ with the restriction of $\tau^\prime\colon \mathcal{M}_{\theta^\prime}(\Gamma^\prime)\to \mathcal{M}_{\theta^\prime_0}(\Gamma^\prime)$ to $E_\ell$, where $\theta^\prime_0\in\Theta_{\Gamma^\prime}$ is determined by the character $\chi_{\theta^\prime_0}:=\res^{G_\Gamma}_{G_{\Gamma'}}(\chi_{\theta_0})$ of $G_{\Gamma^\prime}$. The parameter $\theta_0^\prime$ is in the boundary
 of the chamber containing $\theta^\prime$ since $\Uns(\tau'_-)$ is nonempty, 
 so by Kronheimer~\cite{Kronheimer}, $\tau^\prime$ is the product of $\id_{\mathbb{A}^1}$ with the contraction of at least one
 $(-1)$-curve. It particular, $\tau^\prime\vert_{E_\ell}$ contracts a divisor to a curve, and hence so too does $\tau_-\vert_D$.       
 \end{proof}

 Now, consider the diagram 
 \begin{equation}
\begin{tikzcd}
\label{eqn:flopMtheta}
 \mathcal{M}_{\theta_-} \ar[rr,"\psi",dashed] \ar[dr,"\widetilde{\tau}_-"'] & & \mathcal{M}_{\theta_+} \ar[dl,"\widetilde{\tau}_+"] \\
& \widetilde{\mathcal{M}}_{\theta_0} & 
\end{tikzcd}
\end{equation}
 of schemes over $Y=\mathbb{A}^3/\Gamma$ as in \eqref{eqn:flopStein}, where $\widetilde{\mathcal{M}}_{\theta_0}$ is the normalisation of $\mathcal{M}_{\theta_0}$. Since $\widetilde{\tau}_\pm$ have connected fibres, we may classify GIT walls into four types. Recall from \eqref{eqn:O1pullback} that $L_{C_-}(\theta_0)$ is the semi-ample line bundle that determines the morphism $\widetilde{\tau}_-$. 
 Then either $L_{C_-}(\theta_0)$:
 \begin{itemize}
     \item is ample, in which case $\widetilde{\tau}_-$ is an isomorphism and we say that the wall is of \emph{type $\0$}; or it
    \item defines a class on the boundary of the ample cone of $\mathcal{M}_{\theta_-}$, and since $\theta_0$ is general in the wall, this class lies in the interior of a codimension-one face of $\Amp(\mathcal{M}_{\theta_-}/Y)$ and hence $\widetilde{\tau}_-$ is a primitive contraction. In this case, we say that the wall is:
    \begin{itemize}
        \item  of \emph{type $\I$} if $\widetilde{\tau}_-$ contracts a curve to a point;
        \item of \emph{type $\III$} if $\widetilde{\tau}_-$ contracts a surface to a curve.
    \end{itemize}
 \end{itemize}
 In principle, the morphism $\widetilde{\tau}_-$ might contract a surface to a point - a \emph{type $\II$} contraction - but that surface would necessarily be proper, thereby contradicting Lemma~\ref{lem:nosenior}.
 
 Since $\mathcal{M}_{\theta_-}$ and $\mathcal{M}_{\theta_+}$ are both crepant resolutions of $\mathbb{A}^3/\Gamma$, the type of a wall is independent of whether we replace $\widetilde{\tau}_-$ by $\widetilde{\tau}_+$ throughout the above. In short, the type is independent of the side from which we approach the wall.

 \begin{lem}
 \label{lem:noseniorflipping}
 Let $\Gamma$ satisfy the conditions of Lemma~\ref{lem:nosenior}. For any wall of type $\I$, the unstable loci $\Uns(\tau_-)\subseteq \mathcal{M}_{\theta_-}$ and $\Uns(\tau_+)\subseteq \mathcal{M}_{\theta_+}$ each have codimension at least two.
 \end{lem}
 \begin{proof}
 If $\Uns(\tau_-)$ had an irreducible component $D$ of codimension one, then Lemma~\ref{lem:unstablesurfacescontracted} shows that $\tau_-$ contracts $D$ and hence so does $\widetilde{\tau}_-$. However, $\widetilde{\tau}_-$ contracts only a curve, a contradiction. The $\tau_+$ case is identical.
 \end{proof}
 
 \begin{prop}
 \label{prop:no0orII}
Let $\Gamma$ satisfy the conditions of Lemma~\ref{lem:nosenior}.  There are no GIT walls of type $\0$.
 \end{prop}
 \begin{proof}
 Suppose for a contradiction that chambers $C_-, C_+\subset \Theta$ are separated by a type $\0$ wall. For $\theta_-\in C_-$ and $\theta_+\in C_+$, both $\widetilde{\tau}_-$ and $\widetilde{\tau}_+$ from \eqref{eqn:flopMtheta} are isomorphisms, and hence so is the rational map $\psi$ from \eqref{eqn:flopMtheta}. However, the tautological families agree only on the locus $\mathcal{M}_{\theta_-}\setminus \Uns(\tau_-) \cong \mathcal{M}_{\theta_+}\setminus \Uns(\tau_+)$, otherwise the isomorphic fibres of the tautological families over a strictly $\theta_0$-semistable point would be $\Gamma$-constellations that are $\theta_+$-stable and $\theta_-$-stable in addition to being strictly $\theta_0$-semistable, thereby contradicting Lemma~\ref{lem:theta0stable}\one.
  Since $\mathcal{M}_{\theta_-}\cong \mathcal{M}_{\theta_+}$ is normal, the locus $\Uns(\tau_-)\cong\Uns(\tau_+)$ where the tautological families differ cannot have an irreducible component of codimension at least two, otherwise these tautological families would extend uniquely over that component \cite[Proposition~1.6]{Hartshorne80}, forcing them to agree beyond $\mathcal{M}_{\theta_-}\setminus \Uns(\tau_-) \cong \mathcal{M}_{\theta_+}\setminus \Uns(\tau_+)$. Thus, every irreducible component of $\Uns(\tau_-)\cong\Uns(\tau_+)$ is of codimension one. However, if there were such a component,  Lemma~\ref{lem:unstablesurfacescontracted} shows that it would be contracted by $\widetilde{\tau}_-$, a contradiction. 
 \end{proof}

 In passing, we record the following fact for groups $\Gamma$ that do not satisfy Lemma~\ref{lem:nosenior}.
 
 \begin{lem}
 If a finite subgroup $\Gamma\subset \SL(3,\kk)$ fails to satisfy the conditions from Lemma~\ref{lem:nosenior}, then every chamber $C$ whose closure is strongly convex has a wall of type $\0$.
 \end{lem}
 \begin{proof}
 The linearisation map $L_C$ is surjective by Proposition~\ref{prop:BKR}, so the kernel of $L_C$ must be nonzero by Lemma~\ref{lem:nosenior}. Since $\overline{C}$ is strongly convex, Corollary~\ref{c:git-stronglyconvex} implies that $C$ has a wall such that the morphism $\widetilde{\tau}\colon \mathcal{M}_\theta\to\widetilde{\mathcal{M}}_{\theta_0}$ into the wall is an isomorphism. This wall is of type $\0$.
 \end{proof}

  \subsection{Birational geometry}
 We can now state and prove the main result of this section. 
  
 \begin{thm}
 \label{thm:noseniormain}
 Let $\Gamma\subset \SL(3,\kk)$ satisfy the equivalent conditions from Lemma~\ref{lem:nosenior}. The conclusions of Theorem~\ref{thm:movable} hold for any chamber $C$ and any projective crepant resolution $f\colon X\to Y=\mathbb{A}^3/\Gamma$.
 \end{thm}
 \begin{proof}
 Let $C\subset \Theta$ be any GIT chamber. For $\theta\in C$, we know $\mathcal{M}_\theta$ is smooth by Proposition~\ref{prop:BKR}, and the linearisation map $L_C$ is an isomorphism by Lemma~\ref{lem:nosenior}, so Condition~\ref{cond:GIT}(1) holds for $C$. Next, consider any wall in the interior of the GIT region $R_{C}$ containing $C$. The wall cannot be of type $\0$ or $\II$ by Proposition~\ref{prop:no0orII}, nor can it be of type $\III$ because interior walls of $R_{C}$ must induce small contractions. Therefore, the wall must be of type $\I$, so $\widetilde{\tau}_-$ and $\widetilde{\tau}_+$ each contract a curve to a point. Lemma~\ref{lem:noseniorflipping} shows that every such wall satisfies the assumptions of Proposition~\ref{prop:floppingwall}. It follows that every interior wall of $R_C$ is flipping, so Condition~\ref{cond:GIT}(2) holds. Finally, given a boundary wall of $R_C$, the only possibility left is that the wall is of type $\III$. In particular, the morphism $\widetilde{\tau}_-$ for that wall contracts a (necessarily nonproper) divisor to a curve, so the wall is of divisorial type. Thus, Condition~\ref{cond:GIT}(3) holds for the chamber $C$, so the conclusions of Theorem~\ref{thm:movable} hold for the specific projective crepant resolution $f_\theta\colon \mathcal{M}_\theta\to\mathbb{A}^3/\Gamma$, where $\theta\in C$. These same conclusions must therefore also hold for the chamber $C_X:=L_C^{-1}(\Amp(X/Y))$ that defines $X\cong \mathcal{M}_\theta$ for $\theta\in C_X$.
 \end{proof}

 \begin{remark}
 \label{rem:Wemyss}
 Theorem~\ref{thm:noseniormain} implies in particular that every projective crepant resolution of $\mathbb{A}^3/\Gamma$ is of the form $\mathcal{M}_\theta$ for some generic $\theta$. As noted in the introduction, this statement follows from the work of Wemyss~\cite[Theorem~6.2]{Wemyss18}, which generalised the study of dihedral and trihedral singularities by Nolla de Celis and Sekiya~\cite[Corollaries~1.3 and 1.5]{NollaSekiya17}. Our direct, geometric proof bypasses the algebraic approach via mutation introduced in \cite{Wemyss18}, while our description of the relative movable cone $\Mov(X/Y)$ follows from Theorem~\ref{thm:movable}. In fact, our approach shows that for any chamber $C$, it is not hard to say which should be the next wall to crash through to induce any given flop of $\mathcal{M}_\theta$ for $\theta\in C$: one simply chooses the wall of $\overline{C}$ that's identified by $L_C$ with the given flopping wall of the nef cone of $\mathcal{M}_\theta$ for $\theta\in C$.
\end{remark}

\def\cprime{$'$} \def\cprime{$'$} \def\cprime{$'$} \def\cprime{$'$}
\def\cprime{$'$} \def\cprime{$'$} \def\cprime{$'$} \def\cprime{$'$}
\def\cprime{$'$} \def\cprime{$'$} \def\cprime{$'$} \def\cprime{$'$}
\def\cprime{$'$} \def\cprime{$'$}

 \end{document}